%% file: ms.tex
\documentclass[10pt,letterpaper]{amsart} 

\usepackage{enumitem}
\usepackage{amsmath,amsfonts,amssymb,mathtools}
\usepackage{tikz}
\usepackage{subcaption} 
\usepackage{comment}

\usepackage{amsthm}
\theoremstyle{plain}
\newtheorem{theorem}{Theorem}[section]

\newtheorem{lemma}[theorem]{Lemma}
\newtheorem{corollary}[theorem]{Corollary}
\theoremstyle{definition}
\newtheorem{definition}[theorem]{Definition}
\newtheorem{conjecture}[theorem]{Conjecture}

\theoremstyle{remark}
\newtheorem{remark}[theorem]{Remark}
\newtheorem{example}[theorem]{Example}

\begin{document}

\title[]{Schur positivity of difference of products of derived Schur polynomials}
\author{Julius Ross}
\address{Department of Mathematics, Statistics, and Computer Science, University of Illinois at Chicago, Chicago, IL 60607}
\email{juliusro@uic.edu}
\author{Kuang-Yu Wu}
\address{Department of Mathematics, Statistics, and Computer Science, University of Illinois at Chicago, Chicago, IL 60607}
\email{kwu33@uic.edu}

\begin{abstract}
To any Schur polynomial $s_{\lambda}$ one can associated its derived polynomials $s_{\lambda}{(i)}$ $i=0,\ldots,|\lambda|$ by the rule
$$s_{\lambda}(x_1+t,\ldots,x_n+t) = \sum_i s_{\lambda}^{(i)}(x_1,\ldots,x_n) t^i.$$
We conjecture that $$(s_{\lambda}^{(i)})^2 - s_{\lambda}^{(i-1)} s_{\lambda}^{(i+1)}$$ is always Schur positive and prove this when $i=1$ for rectangles $\lambda = (k^\ell)$, for hooks $\lambda = (k, 1^{\ell -1})$, and when $\lambda = (k,k,1)$ or $\lambda = (3,2^{k-1})$.  
\end{abstract}

\maketitle

\newcommand{\LR}{\operatorname{LR}}

\captionsetup{width=.9\textwidth}

\input{1_Introduction.tex}
\input{2_Prelimimaries.tex}
\input{3_Setup.tex}
\input{4_Hook.tex}
\input{5_kk1.tex}
\input{6_conjugate.tex}

\bibliographystyle{plain}
\bibliography{ref}

\end{document}

%% file: 1_Introduction.tex
\section{Introduction}

Among the many notions of positivity for symmetric polynomials,
one that has attracted particular interest is Schur positivity
\cite{BO14,BBR06,BM04,CWY10,FFLP05,GXY19,KWW08,K04,LLT97,M08,MvW09,MW12}
that demands that a polynomial be written as a non-negative sum of Schur polynomials.
In this paper we explore a notion of Schur-positive log-concavity, which results in new classes of polynomials that we conjecture to be Schur positive.

To describe this, for any partition $\lambda$,
we let $s_{\lambda} (x_1 , \ldots , x_n)$ denote the associated Schur polynomial.
Following \cite{RT23a},
we define the \emph{derived Schur polynomials}
$s_{\lambda}^{(i)}(x_1 , \ldots , x_n)$
for $i=0,\ldots,|\lambda|$ to be the symmetric polynomials required to satisfy
\[
    s_{\lambda}(x_1 + t , \ldots , x_n + t)
    =
    \sum_{i = 0}^{|\lambda|} t^i s_{\lambda}^{(i)}(x_1 , \ldots , x_n)\quad \text{ for all } t .
\]
All derived Schur polynomials are Schur positive by \cite[Theorem 1.5]{CK20} or \cite[Remark 5.3]{RT23b},
and hence products of derived Schur polynomials are also Schur positive.
In this paper, we consider the following:

\begin{conjecture} \label{conjecture:1}
    Let $\lambda$ be a partition.
    Then
    \[
        (s_{\lambda}^{(i)})^2 - s_{\lambda}^{(i - 1)}s_{\lambda}^{(i + 1)}
    \]
    is Schur positive for all $i \geq 1$.
\end{conjecture}

This conjecture can be thought of as a strong form of ``log-concavity'' of the function
$i \mapsto s_{\lambda}^{(i)}$.
It is motivated by the following weaker statement
\cite[Corollary 10.12]{RT23b},
which states that if $x_1 , \ldots , x_n$ are non-negative,
then $i \mapsto s_{\lambda}^{(i)} (x_1 , \ldots , x_n)$
is a log concave sequence of real numbers.
Clearly this would be implied by Conjecture \ref{conjecture:1}
since Schur polynomials are monomial positive.

An easy case in which Conjecture \ref{conjecture:1} is clear is where
$\lambda = (1^n) = (\overbrace{1 , \ldots , 1}^{n \text{ copies}})$
and there are exactly $n$ variables.
Since $s_{(1^n)}$ is the $n$-th elementary symmetric polynomial $e_n$,
the derived Schur polynomials $s^{(i)}_{(1^n)}$ are equal to $e_{n - i}$.
Therefore,
\[
    (s_{\lambda}^{(i)})^2 - s_{\lambda}^{(i - 1)} s_{\lambda}^{(i + 1)}
    =
    \begin{vmatrix}
        e_{n-i} & e_{n-i+1} \\
        e_{n-i-1} & e_{n-i}
    \end{vmatrix}
    =
    s_{(n-i,n-i)'}
    =
    s_{(2^{n-i})}
\]
by the Jacobi-Trudi identity.

Our main result is the following.

\begin{theorem}[$=$ Corollary \ref{corollary:rectangle}, Theorem \ref{theorem:hook}, Theorem \ref{theorem:kk1}, Theorem \ref{theorem:conjugate}] \label{theorem:main}
    The polynomial
    \[
        (s_{\lambda}^{(1)})^2 - s_{\lambda}s_{\lambda}^{(2)}
    \]
    is Schur positive if $\lambda$ is a partition of one of the following forms.
    \begin{enumerate}[label = (\alph*)]
        \item Rectangles, i.e. $\lambda = (k^{\ell})$, where $k , \ell$ are positive integers.
        \item Hooks, i.e. $\lambda = (k , 1^{\ell - 1})$, where $k , \ell \geq 2$
        \item $\lambda = (k , k , 1)$, where $k \geq 3$.
        \item $\lambda = (3 , 2^{k - 1})$, where $k \geq 3$.
    \end{enumerate}
\end{theorem}

\subsection{Outline of the proof}

Fix a partition $\lambda$.  Using work of Corteel-Kim \cite{CK20}, the $i$-th derived Schur polynomials $s_{\lambda}^{(i)}$ can be expressed explicitly as a linear combinations of Schur polynomials corresponding to partitions of $|\lambda| - i$ contained in $\lambda$.  For partitions of $|\lambda| - 1$ contained in $\lambda$,
we use $\lambda_{(j)}$ to denote the partition obtained by
removing the last cell from the $j$-th row of the Young diagram of $\lambda$.
For partitions of $|\lambda| - 2$ contained in $\lambda$,
we split them into three different types depending on the configuration of the two cells removed from the Young diagram of $\lambda$.
(See Figure \ref{figure:3 types}.)
\begin{enumerate}[label = (\roman*)]
    \item A type I partition $\lambda_{(j , k)}$ (of $|\lambda| - 2$ contained in $\lambda$) for $j<k$ is obtained by removing two nonadjacent cells from the $j$-th and the $k$-th rows.
    \item A type II partition $\lambda_{(j , \uparrow)}$ is obtained by removing the last cell from the $j$-th row and the cell above it.
    \item A type III partition $\lambda_{(j , \leftarrow)}$ is obtained by removing the last two cells from the $j$-th row.
\end{enumerate}

\begin{figure}
    \captionsetup{width=300pt}
    \caption{Examples of three different types of partitions of $|\lambda| - 2$, where $\lambda = (5 , 3 , 3 , 2)$}
    \label{figure:3 types}
    \begin{subfigure}{0.3\textwidth}
        \caption{$\lambda_{(1 , 4)} = (4 , 3 , 3 , 1)$}
        \vspace{5pt}
        \begin{center} \begin{tikzpicture}[scale=0.4]
            \draw
                (0 , 0) -- (4 , 0)
                (0 , -1) -- (4 , -1)
                (0 , -2) -- (3 , -2)
                (0 , -3) -- (3 , -3)
                (0 , -4) -- (1 , -4)
                (0 , 0) -- (0 , -4)
                (1 , 0) -- (1 , -4)
                (2 , 0) -- (2 , -3)
                (3 , 0) -- (3 , -3)
                (4 , 0) -- (4 , -1)
            ;
            \draw[dotted]
                (4 , 0) -- (5 , 0) -- (5 , -1) -- (4 , -1)
                (2 , -3) -- (2 , -4) -- (1 , -4)
            ;
        \end{tikzpicture} \end{center}
    \end{subfigure}
    \begin{subfigure}{0.3\textwidth}
        \caption{$\lambda_{(3 , \uparrow)} = (5,2,2,2)$}
        \vspace{5pt}
        \begin{center} \begin{tikzpicture}[scale=0.4]
            \draw
                (0 , 0) -- (5 , 0)
                (0 , -1) -- (5 , -1)
                (0 , -2) -- (2 , -2)
                (0 , -3) -- (2 , -3)
                (0 , -4) -- (2 , -4)
                (0 , 0) -- (0 , -4)
                (1 , 0) -- (1 , -4)
                (2 , 0) -- (2 , -4)
                (3 , 0) -- (3 , -1)
                (4 , 0) -- (4 , -1)
                (5 , 0) -- (5 , -1)
            ;
            \draw[dotted]
                (3 , -1) -- (3 , -2) -- (2 , -2)
                (3 , -2) -- (3 , -3) -- (2 , -3)
		;
	\end{tikzpicture} \end{center}
    \end{subfigure}
    \begin{subfigure}{0.3\textwidth}
        \caption{$\lambda_{(1 , \leftarrow)} = (3 , 3 , 3 , 2)$}
        \vspace{5pt}
        \begin{center} \begin{tikzpicture}[scale=0.4]
            \draw
                (0 , 0) -- (3 , 0)
                (0 , -1) -- (3 , -1)
                (0 , -2) -- (3 , -2)
                (0 , -3) -- (3 , -3)
                (0 , -4) -- (2 , -4)
                (0 , 0) -- (0 , -4)
                (1 , 0) -- (1 , -4)
                (2 , 0) -- (2 , -4)
                (3 , 0) -- (3 , -3)
            ;
            \draw[dotted]
                (4 , 0) -- (5 , 0) -- (5 , -1) -- (4 , -1)
                (3 , 0) -- (4 , 0) -- (4 , -1) -- (3 , -1)
            ;
        \end{tikzpicture} \end{center}
    \end{subfigure}
\end{figure}
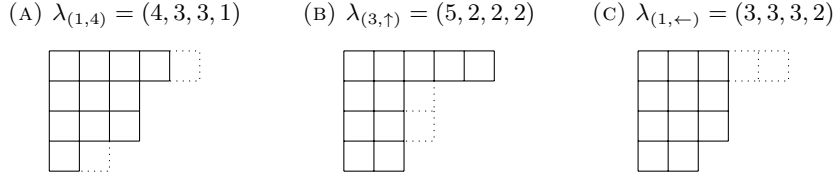

Then, we rewrite $(s_{\lambda}^{(1)})^2 - s_{\lambda}s_{\lambda}^{(2)}$
as a non-negative linear combination of three classes of polynomials 
\begin{align}
2 s_{\lambda_{(j)}}s_{\lambda_{(k)}} - s_{\lambda}s_{\lambda_{(j,k)}}\tag{Type I}\\
s_{\lambda_{(j)}}^2 -s_{\lambda}s_{\lambda_{(j,\uparrow)}}\tag{Type II}\\
s_{\lambda_{(j)}}^2 - s_{\lambda}s_{\lambda_{(j,\leftarrow)}}\tag{Type III}\end{align}
Now Type II and Type III polynomials are Schur positive by \cite{LPP07}, from which part (a) of Theorem \ref{theorem:main} immediately follows
as Type I partitions do not exist for rectangles.  However as Type I partitions are not Schur positive in general, more work must be done for parts (b)-(d).

For partitions in (b) and (c) of Theorem \ref{theorem:main},
we proceed by showing that the coefficient of
$s_{\alpha}$ in $2 \, s_{\lambda_{(j)}}s_{\lambda_{(k)}} - s_{\lambda}s_{\lambda_{(j,k)}}$
is non-negative for most partitions $\alpha$ of $2|\lambda| - 2$
and identify the partitions for which the coefficients are negative. This is done by exhibiting maps between sets of Littlewood-Richardson tableaux that are at most $2$-to-$1$.
For the partitions with negative coefficients,
we compute their coefficients in $(s_{\lambda}^{(1)})^2 - s_{\lambda}s_{\lambda}^{(2)}$ directly
by enumerating Littlewood-Richardson tableaux
and show that they are non-negative.

Finally, part (d) of Theorem \ref{theorem:main} is obtained by taking conjugates of partitions in (c).

\subsection{Discussion}

More generally, for any homogeneous symmetric polynomial $p$ in $n$ variables,
we may define $p^{(i)}$ for $i=0,\ldots, \deg p$ 
to be the symmetric polynomials required to satisfy
\[
    p(x_1 + t , \ldots , x_n + t)
    =
    \sum_{i = 0}^{\deg p} t^i p^{(i)}(x_1 , \ldots , x_n)
    \text{ for all } t.
\]

\begin{example}
The following example shows
that $p$ being Schur positive is not enough to ensure that $(p^{(1)})^2 - p p^{(2)}$ is Schur positive.    Consider $p = s_{(3)} + s_{(1,1,1)}$ with $n = 3$ variables.
    We have 
    $p^{(1)} = 5 s_{(2)} + s_{(1,1,1)}$
    and
    $p^{(2)} = 11 s_{(1)}$.
    Thus, $(p^{(1)})^2 - p p^{(2)}$ equals
    \begin{multline*}
        (25 s_{(4)} + 35 s_{(3,1)} + 26 s_{(2,2)} + 11 s_{(2,1,1)} + s_{(1,1,1,1)}) - 11 (s_{(4)} + s_{(3,1)} + s_{(2,1,1)} + s_{(1,1,1,1)})
        \\
        = 14 s_{(4)} + 24 s_{(3,1)} + 26 s_{(2,2)} - 10 s_{(1,1,1,1)} \, ,
    \end{multline*}
    which is not Schur positive.
    We remark that this example also shows that  the set of $p$ for which $(p^{(1)})^2 - p p^{(2)}$ is Schur positive is not convex.
\end{example}

Nevertheless we expect that Conjecture \ref{conjecture:1} holds even more generally.

\begin{conjecture} \label{conjecture:product}
    Let $\lambda , \mu$ be partitions
    and define $p = s_{\lambda} s_{\mu}$.
    Then
    \[
        (p^{(i)})^2 - p^{(i - 1)}p^{(i + 1)}
    \]
    is Schur positive for all $i \geq 1$.
\end{conjecture}

Using SageMath,
we have verified that Conjecture \ref{conjecture:1} is true for all partitions $\lambda$ with $|\lambda| \leq 15$ in the cases of at most $100$ variables,
and that Conjecture \ref{conjecture:product} is true for all pairs of partitions $(\lambda , \mu)$ with $|\lambda| + |\mu| \leq 10$ in the cases of at most $100$ variables.

\begin{remark}
    If $p$ and $q$ are symmetric polynomials such that
    \[
        (p^{(1)})^2 - p p^{(2)} \; , \; (q^{(1)})^2 - q q^{(2)}
    \]
    are both Schur positive, then $pq$ satisfies that
    \[
        ((pq)^{(1)})^2 - (pq) (pq)^{(2)}
    \]
    is Schur positive.  To this observe that since we have
    \begin{align*}
        (pq)^{(1)} &
        = p^{(1)} q + p q^{(1)} \, , \\
        (pq)^{(2)} &
        = p^{(2)} q + p^{(1)} q^{(1)} + p q^{(2)} \, ,
    \end{align*}
    we obtain
    \begin{align*}
        & ((pq)^{(1)})^2 - (pq) (pq)^{(2)} \\
        = \; & \big( (p^{(1)})^2 - p p^{(2)} \big) q^2 + p p^{(1)} q q^{(1)} + p^2 \big( (q^{(1)})^2 - q q^{(2)} \big) \, .
    \end{align*}
    This is Schur positive as derived Schur polynomials are Schur positive and Schur positivity is preserved under taking products.
\end{remark}


\subsection*{Acknowledgments} JR thanks Matt Larson for bringing up this Schur positive question and for other stimulating discussions, as well as Allen Knutson for discussion on this topic.  This material is based upon work supported by the National Science Foundation under Grant No. DMS-1749447.

%% file: 2_Prelimimaries.tex
\section{Preliminaries}

We list some standard definitions and facts regarding Schur polynomials and Littlewood-Richardson coefficients that will be needed in this article.
We refer readers to \cite{Ful97,Sta24} for more details on this subject.

\subsection{Partitions and Schur polynomials}

A partition of a positive integer $d$ is a non-increasing sequence of non-negative integers
$\lambda = (\lambda_1 , \lambda_2 , \lambda_3 , \ldots )$
with $|\lambda| := \sum_{i = 1}^{\infty} \lambda_i = d$.
Each entry $\lambda_i$ of the partition $\lambda$ is called a part of $\lambda$.
The length of a partition $\lambda$,
denoted by $\ell(\lambda)$,
is the number of non-zero parts of $\lambda$.
The zeros are often omitted when writing down a partition,
i.e. if $\ell = \ell(\lambda)$,
we write $\lambda = (\lambda_1 , \ldots , \lambda_{\ell})$.
We use the shorthand $a^f$ for
$a , \ldots , a$ ($f$ times) in a partition. 
For example,
$(5 , 3^2 , 2) = (5 , 3 , 3 , 2)$
is a partition of $13$ with length $4$.

Each partition $\lambda = (\lambda_1 , \ldots , \lambda_{\ell})$
corresponds to a Young diagram,
which is a left-justified array of square cells
with exactly $\lambda_i$ cells in the $i$-th row for $1 \leq i \leq \ell$.
For example,
Figure \ref{figure:Young diagram (5332)} is the Young diagram of $(5 , 3^2 , 2)$.
We use $\lambda'$ to denote the conjugate partition of $\lambda$,
which is defined to be the partition corresponding to the reflection of the Young diagram of $\lambda$ in the diagonal line.

\begin{figure}
    \caption{Young diagram of $(5 , 3^2 , 2)$}
    \label{figure:Young diagram (5332)}
    \vspace{5pt}
    \begin{tikzpicture}[scale=0.4]
        \draw
            (0 , 0) -- (5 , 0)
            (0 , -1) -- (5 , -1)
            (0 , -2) -- (3 , -2)
            (0 , -3) -- (3 , -3)
            (0 , -4) -- (2 , -4)
            (0 , 0) -- (0 , -4)
            (1 , 0) -- (1 , -4)
            (2 , 0) -- (2 , -4)
            (3 , 0) -- (3 , -3)
            (4 , 0) -- (4 , -1)
            (5 , 0) -- (5 , -1)
        ;
    \end{tikzpicture}
\end{figure}
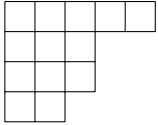

We write $\mu \subseteq \lambda$ for two partitions $\mu , \lambda$
if $\mu_i \leq \lambda_i$ for all $i \geq 1$,
or equivalently,
if the Young diagram of $\mu$ is contained in that of $\lambda$.
In this case,
we define the skew shape $\lambda / \mu$
to be the diagram obtained by
removing the Young diagram of $\mu$ from that of $\lambda$.
For example,
Figure \ref{figure:skew shape (5332)/(331)} is the skew shape $(5 , 3^2 , 2) / (3^2 , 1)$, and going forward we will illustrate skew shapes in this way with white cells and gray shaded region.

\begin{figure}
    \caption{The skew shape $(5 , 3^2 , 2) / (3^2 , 1)$}
    \label{figure:skew shape (5332)/(331)}
    \vspace{5pt}
    \begin{tikzpicture}[scale=0.4]
        \fill[gray] (0 , 0) -- (3 , 0) -- (3 , -2) -- (1 , -2) -- (1 , -3) -- (0 , -3) -- (0 , 0) ;
        \draw
            (3 , 0) -- (5 , 0)
            (3 , -1) -- (5 , -1)
            (1 , -2) -- (3 , -2)
            (0 , -3) -- (3 , -3)
            (0 , -4) -- (2 , -4)
            (0 , -3) -- (0 , -4)
            (1 , -2) -- (1 , -4)
            (2 , -2) -- (2 , -4)
            (3 , -2) -- (3 , -3)
            (3 , 0) -- (3 , -1)
            (4 , 0) -- (4 , -1)
            (5 , 0) -- (5 , -1)
        ;
    \end{tikzpicture}
\end{figure}
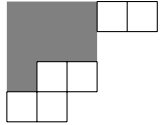

Given a partition $\lambda$ and a positive integer $n \geq \ell(\lambda)$,
we let $s_{\lambda} \in \mathbb{Z} [x_1 , \ldots , x_n]$ be the Schur polynomial of the partition $\lambda$ in $n$ variables.
By the Jacobi-Trudi identities, we have
\[
    s_{\lambda} = \det (h_{\lambda_i + j - i})_{i,j = 1}^{\ell(\lambda)} = \det(e_{\lambda'_i + j - i})_{i,j = 1}^{\ell(\lambda')}\, ,
\]
where $h_k$ is the $k$-th complete homogeneous symmetric polynomial 
and $e_k$ is the $k$-th elementary homogeneous symmetric polynomial.

Each Schur polynomial $s_{\lambda}$ is a homogeneous symmetric polynomial of degree $|\lambda|$.
The vector space $\Lambda_d$ of (rational coefficient)  homogeneous symmetric polynomial of a fixed degree $d$ has a basis given by $\{s_{\lambda}\}_{\lambda : |\lambda| = d}$.
We say that a symmetric polynomial $f \in \Lambda_d$ is Schur positive,
and write $f \geq_s 0$,
if the coefficients in the expression of $f$ as a linear combination of $s_{\lambda}$ are all non-negative.

\subsection{Littlewood-Richardson rule}

Given two partitions $\mu , \nu$,
the product $s_{\mu} s_{\nu}$ of Schur polynomials is a homogeneous symmetric polynomial
and can be written as
\[
    s_{\mu} s_{\nu} = \sum_{\lambda : |\lambda| = |\mu| + |\nu|} c_{\mu \nu}^{\lambda} s_{\lambda} \, .
\]
The integers $c_{\mu \nu}^{\lambda}$ are called the Littlewood-Richardson coefficients.

\begin{definition} \label{definition:Littlewood-Richardson tableaux}
    A Littlewood-Richardson tableau $T$ of shape $\lambda/\mu$ is a filling of boxes of the skew shape $\lambda/\mu$ with positive integers that satisfies the following.
    \begin{enumerate}[label=(\alph*)]
        \item The filling is weakly increasing across each row, and strictly increasing down each column.
        \item The reverse reading word $w = (w_1 , \ldots , w_{|\lambda| - |\mu|})$ of $T$, i.e.\ the sequence of positive integers obtained by reading the entries of $T$ from right to left and top row to bottom row, is a lattice permutation, i.e.\ in every initial part $(w_1 , \ldots , w_j)$ of $w$ the number of any positive integer $i$ is not less than the number of $i + 1$.
    \end{enumerate}
\end{definition}

The content of a Littlewood-Richardson tableau $T$ is defined as a sequence of integer
$\nu = (\nu_i)_{i = 1}^{\infty}$
such that there are exactly $\nu_i$ of $i$'s in $T$.
It follows from condition (b) of Definition \ref{definition:Littlewood-Richardson tableaux} that the content of a Littlewood-Richardson tableau is always a partition.

The Littlewood-Richardson rule states that
the Littlewood-Richardson coefficient $c_{\mu \nu}^{\lambda}$ is equal to
the number of Littlewood-Richardson tableaux of shape $\lambda/\mu$ with content $\nu$.
It follows that all Littlewood-Richardson coefficients $c_{\mu \nu}^{\lambda}$ are non-negative,
and hence products of Schur polynomials are Schur positive.
By the fact that $c_{\mu \nu}^{\lambda} = c_{\nu \mu}^{\lambda}$ and the Littlewood-Richardson rule,
it is also clear that $c_{\mu \nu}^{\lambda} = 0$ if $\mu \not \subseteq \lambda$ or $\nu \not \subseteq \lambda$.

For a skew shape $\lambda / \mu$,
if $\lambda_i > \mu_i$
then we call the $i$-th row of $\lambda / \mu$ a non-empty row.
The following simple criteria for non-existence of Littlewood-Richardson tableaux will be useful.

\begin{lemma} \label{lemma:badLRT}
    Littlewood-Richardson tableaux of shape $\lambda/\mu$ with content $\nu$ do not exist,
    and hence $c^{\lambda}_{\mu\nu} = 0$
    if any of the following statements is true.
    \begin{enumerate}[label=(\alph*)]
        \item There exists a column of $\lambda/\mu$ with more than $\ell(\nu)$ cells.
        \item There are more than $\nu_1$ cells in the first non-empty row of $\lambda/\mu$.
        \item The skew shape $\lambda / \mu$ has fewer than $\ell(\nu)$ non-empty rows.
    \end{enumerate}
\end{lemma}

\begin{proof}
    Part (a) follows from the requirement for strictly increasing columns.
    Suppose the shape $\lambda/\mu$ has a column with $m$ cells, where $m > \ell(\nu)$.
    Then the integer placed in the last cell of this column must be at least $m$.
    However, we have $\mu_i = 0$ for $i \geq m$ as $m > \ell(\nu)$.
    Thus, Littlewood-Richardson tableaux of shape $\lambda/\mu$ with content $\nu$ do not exist.
    
    Parts (b) and (c) are corollaries of the following more general lemma
    (by taking $k = 1$ and $k$ equals the number of non-empty rows in $\lambda/\mu$, respectively).
\end{proof}

\begin{lemma}
    In a Littlewood-Richardson tableau $T$ of shape $\lambda/\mu$ with content $\nu$,
    for every positive integer $k \leq \ell(\nu)$,
    the $k$-th non-empty row can only contain integers $1 , \ldots , k$.
    In particular, in order for a Littlewood-Richardson tableau of shape $\lambda/\mu$ with content $\nu$ to exist,
    the total number of cells in the first $k$ non-empty rows of $\lambda / \mu$ cannot exceed $\sum_{i = 1}^k \nu_i$
    for every positive integer $k \leq \ell(\nu)$.
\end{lemma}

\begin{proof}
    Suppose the first statement is false.
    Let $k_0$ be the smallest integer such that the statement fails,
    and let $A$ be the integer in the last cell of the $k_0$-th non-empty row of $T$.
    The integer $A$ is the largest in this row,
    so we have $A > k_0$.
    However, the first $(k_0 - 1)$ non-empty rows of $T$ only contain integers $1 , \ldots , k_0 - 1$.
    This implies that $A$ precedes all $k_0$'s in the reverse reading word of $T$,
    which contradicts condition (b) in Definition \ref{definition:Littlewood-Richardson tableaux}.

    Now, since the first $k$ non-empty rows of a Littlewood-Richardson tableau $T$ of shape $\lambda/\mu$ with content $\nu$ can only contain integers $1 , \ldots , k$,
    the total number of cells in the first $k$ non-empty rows cannot exceed the total number of integers $1 , \ldots , k$ in $T$,
    which is exactly $\sum_{i = 1}^k \nu_i$.
\end{proof}

%% file: 3_Setup.tex
\section{Setup}

Fix a partition $\lambda$.
We first define notations for partitions of $|\lambda| - 1$ and $|\lambda| - 2$ that are contained in $\lambda$.

\begin{definition}
    Let $J(\lambda) := \{ j : \lambda_{j} > \lambda_{j+1} \}$.
    \begin{enumerate}[label = (\alph*)]
        \item
        For each $j \in J(\lambda)$, define $\lambda_{(j)}$ to be the partition corresponding to the Young diagram obtained by removing the last cell of the $j$-th row from the Young diagram of $\lambda$;
        i.e.\ $\lambda_{(j)} = (\lambda_{(j),i})$ is defined by $\lambda_{(j),j} = \lambda_{j} - 1$ and $\lambda_{(j),i} = \lambda_{i}$ for $i \neq j$.
        \item
        For each pair of $j , k \in J(\lambda)$ with $j < k$, define $\lambda_{(j,k)}$ to be the partition corresponding to the Young diagram obtained by removing the last cell of the $j$-th row and the last cell of the $k$-th row from the Young diagram of $\lambda$;
        i.e.\ $\lambda_{(j,k)} = (\lambda_{(j,k),i})$ is defined by $\lambda_{(j,k),j} = \lambda_{j} - 1$, $\lambda_{(j,k),k} = \lambda_{k} - 1$ and $\lambda_{(j,k),i} = \lambda_{i}$ for $i$ not equal to $j$ or $k$.
        
        \item
        Given $j \in J(\lambda)$ with $\lambda_{j-1} = \lambda_{j}$, define $\lambda_{(j,\uparrow)}$ to be the partition corresponding to the Young diagram obtained by removing the last cell of the $j$-th row and the cell directly above it from the Young diagram of $\lambda$;
        i.e.\ $\lambda_{(j,\uparrow)} = (\lambda_{(j,\uparrow),i})$ is defined by $\lambda_{(j,\uparrow),j} = \lambda_{(j,\uparrow),j-1} = \lambda_{j} - 1$ and $\lambda_{(j,\uparrow),i} = \lambda_{i}$ for $i$ not equal to $j$ or $j-1$.
        
        \item
        Given $j \in J(\lambda)$ with $\lambda_{j} - \lambda_{j+1} > 1$, define $\lambda_{(j,\leftarrow)}$ to be the partition corresponding to the Young diagram obtained by removing the last two cells of the $j$-th row from the Young diagram of $\lambda$;
        i.e.\ $\lambda_{(j,\leftarrow)} = (\lambda_{(j,\leftarrow),i})$ is defined by $\lambda_{(j,\leftarrow),j} = \lambda_{j} - 2$ and $\lambda_{(j,\leftarrow),i} = \lambda_{i}$ for $i \neq j$.
        
        \item
        For convenience, we also define $s_{\lambda_{(j,\uparrow)}} := 0$ if $\lambda_{j-1} > \lambda_{j}$ (i.e.\ if $\lambda_{(j,\uparrow)}$ does not exist),
        and define $s_{\lambda_{(j,\leftarrow)}} := 0$ if $\lambda_{j} - \lambda_{j+1} = 1$ (i.e.\ if $\lambda_{(j,\leftarrow)}$ does not exist).
	\end{enumerate}
\end{definition}

\begin{example} \label{example:derived Schur}
    Consider the partition $\lambda = (3,3,2)$.
    We have $J(\lambda) = \{ 2 , 3 \}$.
    There are two partitions of $|\lambda| - 1$ contained in $\lambda$,
    given by $\lambda_{(2)} = (3,2,2)$ and $\lambda_{(3)} = (3,3,1)$,
    and three partitions of $|\lambda| - 2$ that are contained in $\lambda$,
    given by $\lambda_{(2,3)} = (3,2,1)$, $\lambda_{(2,\uparrow)} = (2,2,2)$, and $\lambda_{(3,\leftarrow)} = (3,3)$.
    We also have $s_{\tau_{(3,\uparrow)}} = s_{\tau_{(2,\leftarrow)}} = 0$.
\end{example}

Let $n$ be the numbers of variables.
Note that $s_{\lambda} (x_1 , \ldots , x_n) = 0$ whenever $n < \ell(\lambda)$
by Jacobi-Trudi identity and the fact that $e_m (x_1 , \ldots , x_n) = 0$ whenever $m > n$.
Thus, without loss of generality, we may assume that $n \geq \ell(\lambda)$.

By \cite[Theorem 1.1 and Theorem 1.5]{CK20},
we have
\begin{equation} \tag{$\dagger$} \label{equation:derived Schur}
    \begin{split}
            s_{\lambda}^{(1)} =
        &
            \sum_{j \in J} (n + \lambda_j - j) \, s_{\lambda_{(j)}}
            \, ,
        \\
            s_{\lambda}^{(2)} =
        &
            \sum_{\substack{j,k \in J \\ j<k}} (n + \lambda_j - j) (n + \lambda_k - k)
            \, s_{\lambda_{(j,k)}}
        \\
        &
            + \sum_{j \in J} \frac{1}{2} (n + \lambda_j - j) (n + \lambda_j - j + 1)
            \, s_{\lambda_{(j,\uparrow)}}
        \\
        &
            + \sum_{j \in J} \frac{1}{2} (n + \lambda_j - j) (n + \lambda_j - j - 1)
            \, s_{\lambda_{(j,\leftarrow)}}
            \, .
    \end{split}
\end{equation}

\begin{example}
    Following Example \ref{example:derived Schur}, we have
        \begin{align*}
                s_{(3,3,2)}^{(1)} &= (n + 1) s_{(3,2,2)} + (n - 1) s_{(3,3,1)}
                \, ,
            \\
                s_{(3,3,2)}^{(2)} &= (n + 1)(n - 1) s_{(3,2,1)} + \frac{1}{2} (n + 1)(n + 2) s_{(2,2,2)} + \frac{1}{2} (n - 1)(n - 2) s_{(3,3)}
                \, .
        \end{align*}
    
\end{example}

Plug (\ref{equation:derived Schur}) into $(s_{\lambda}^{(1)})^2 - s_{\lambda} s_{\lambda}^{(2)}$,
and we obtain the following.
\begin{equation} \tag{$*$} \label{equation:main}
    \begin{split}
            (s_{\lambda}^{(1)})^2 - s_{\lambda} s_{\lambda}^{(2)} =
        &
            \sum_{\substack{j,k \in J \\ j<k}}
            (n + \lambda_j - j) (n + \lambda_k - k)
            \, (2 \, s_{\lambda_{(j)}}s_{\lambda_{(k)}} - s_{\lambda}s_{\lambda_{(j,k)}})
	\\
        &
            + \sum_{j \in J}
            \frac{1}{2} (n + \lambda_j - j) (n + \lambda_j - j + 1)
            \, (s_{\lambda_{(j)}}^2 - s_{\lambda}s_{\lambda_{(j,\uparrow)}})
        \\
        &
            + \sum_{j \in J}
            \frac{1}{2} (n + \lambda_j - j) (n + \lambda_j - j - 1)
            \, (s_{\lambda_{(j)}}^2 - s_{\lambda}s_{\lambda_{(j,\leftarrow)}})
    \end{split}
\end{equation}

The second and the third term in (\ref{equation:main}) are always Schur positive by the following lemma.

\begin{lemma} \label{lemma:type 2&3}
    The polynomials $s_{\lambda_{(j)}}^2 - s_{\lambda}s_{\lambda_{(j,\uparrow)}}$ and $s_{\lambda_{(j)}}^2 - s_{\lambda}s_{\lambda_{(j,\leftarrow)}}$ are Schur positive.
\end{lemma}

Our proof of this lemma relies on results in \cite{LPP07} that we record here.
\begin{definition}
    Let $\lambda = (\lambda_i)_{i = 1}^{\infty}$ and $\mu = (\mu_i)_{i = 1}^{\infty}$ be two partitions.
    \begin{enumerate}[label = (\alph*)]
        \item If $\lambda_i + \mu_i$ is even for all $i$, define $\frac{\lambda + \mu}{2} := \big(\frac{\lambda_i + \mu_i}{2}\big)_{i = 1}^{\infty}$.

        \item Let $\nu_1 \geq \nu_2 \geq \nu_3 \geq \cdots$ be the sequence obtained by collecting all parts of $\lambda$ and $\mu$ and reordering them in non-increasing order.
        Define two partitions
        $\operatorname{sort}_1(\lambda , \mu) := (\nu_1 , \nu_3 , \nu_5 , \ldots)$
        and
        $\operatorname{sort}_2(\lambda , \mu) := (\nu_2 , \nu_4 , \nu_6 , \ldots)$.
    \end{enumerate}
\end{definition}

\begin{theorem} \label{theorem:LPP07}
    Let $\lambda = (\lambda_i)_{i = 1}^{\infty}$ and $\mu = (\mu_i)_{i = 1}^{\infty}$ be two partitions.
    Then we have
    \begin{enumerate}[label = (\alph*)]
        \item $\big( s_{\frac{\lambda + \mu}{2}} \big)^2 \geq_{s} s_{\lambda} s_{\mu}$ if $\lambda_i + \mu_i$ is even for all $i$.

        \item $s_{\operatorname{sort}_1(\lambda , \mu)} s_{\operatorname{sort}_2(\lambda , \mu)} \geq_{s} s_{\lambda} s_{\mu}$.
    \end{enumerate}
\end{theorem}

\begin{proof}
    Part (a) is a special case of \cite[Conjectures 1]{LPP07} with $\mu = \rho = 0$, and part (b) is \cite[Conjectures 2]{LPP07}, both of which are true by \cite[Theorem 4]{LPP07}.
\end{proof}

\begin{proof}[Proof of Lemma \ref{lemma:type 2&3}]
    If $\lambda_{(j,\uparrow)}$ or $\lambda_{(j,\leftarrow)}$ do not exist,
    then we have $s_{\lambda_{(j,\uparrow)}} = 0$ or $s_{\lambda_{(j,\leftarrow)}} = 0$
    and the lemma follows from the fact that
    products of Schur polynomials are Schur positive.

    Suppose that $\lambda_{(j,\uparrow)}$ exists.
    We have $\operatorname{sort}_1(\lambda , \lambda_{(j,\uparrow)}) = \operatorname{sort}_2(\lambda , \lambda_{(j,\uparrow)}) = \lambda_{(j)}$,
    and hence $s_{\lambda_{(j)}}^2 - s_{\lambda}s_{\lambda_{(j,\uparrow)}}$ is Schur positive by part (b) of Theorem \ref{theorem:LPP07}.

    Suppose that $\lambda_{(j,\leftarrow)}$ exists.
    We have $\frac{1}{2} (\lambda + \lambda_{(j,\leftarrow)}) = \lambda_{(j)}$,
    and hence $s_{\lambda_{(j)}}^2 - s_{\lambda}s_{\lambda_{(j,\leftarrow)}}$ is Schur positive by part (a) of Theorem \ref{theorem:LPP07}.
\end{proof}

Since the first term in (\ref{equation:main}) does not exist for rectangles $\lambda = (k^{\ell})$,
Lemma \ref{lemma:type 2&3} immediately implies the following.

\begin{corollary} \label{corollary:rectangle}
    Let $\lambda$ be a rectangle,
    i.e.\ a partition of the form $(k^{\ell})$,
    where $k , \ell$ are positive integers.
    Then $(s_{\lambda}^{(1)})^2 - s_{\lambda}s_{\lambda}^{(2)}$ is Schur positive.
\end{corollary}

The first term in (\ref{equation:main}),
however,
is not necessarily Schur positive.
In fact,
since $\operatorname{sort}_1(\lambda_{(j)} , \lambda_{(k)})= \lambda$ and $\operatorname{sort}_2(\lambda_{(j)} , \lambda_{(k)}) = \lambda_{(j,k)}$ in the sense of \cite{LPP07},
we have that
$s_{\lambda_{(j)}}s_{\lambda_{(k)}} - s_{\lambda}s_{\lambda_{(j,k)}}$ is Schur negative
by \cite[Conjecture 2]{LPP07}
(which is true by \cite[Theorem 4]{LPP07}).
In each of the remaining sections,
we pin down the negative contributions in the first term in (\ref{equation:main})
and show that they are compensated by the other two terms in (\ref{equation:main}).

%% file: 4_Hook.tex
\section{Hooks}

Fix positive integers $k , \ell \ge 3$ and consider the partition $\lambda = (k , 1^{\ell - 1})$.
Such a partition is called a hook because of the shape of its Young diagram.
Let $n \geq \ell$ be the number of variables.
The goal of this section is to prove the following.

\begin{theorem} \label{theorem:hook}
    Let $\lambda$ be a hook,
    i.e.\ a partition of the form $(k , 1^{\ell - 1})$, where $k , \ell \geq 3$.
    Then $(s_{\lambda}^{(1)})^2 - s_{\lambda}s_{\lambda}^{(2)}$ is Schur positive.
\end{theorem}

There are two partitions of $|\lambda| - 1$ contained in $\lambda$,
given by
$\lambda_{(1)} = (k - 1 , 1^{\ell - 1})$ and $\lambda_{(\ell)} = (k , 1^{\ell - 2})$,
as well as three partitions of $|\lambda| - 2$ contained in $\lambda$,
given by
$\lambda_{(1,\ell)} = (k - 1 , 1^{\ell - 2})$ and
$\lambda_{(\ell,\uparrow)} = (k , 1^{\ell - 3})$ and
$\lambda_{(1,\leftarrow)} = (k - 2 , 1^{\ell - 1})$.
The theorem will follow from the following two lemmas.

\begin{lemma} \label{lemma:hook type 1}
    Let $\alpha$ be a partition of $2|\lambda| - 2 = 2(k + \ell - 2)$, containing $\lambda$ and not equal to $(k^2 , 2^{\ell - 2})$.
    Then we have $c^{\alpha}_{\lambda_{(1)} , \lambda_{(\ell)}} \geq c^{\alpha}_{\lambda , \lambda_{(1,\ell)}}$.
\end{lemma}

\begin{lemma} \label{lemma:hook special}
    Let $\beta = (k^2 , 2^{\ell - 2})$.
    Then we have
    \[
        \Big(
        c^{\beta}_{\lambda_{(1)} , \lambda_{(\ell)}} ,
        c^{\beta}_{\lambda_{(\ell)} , \lambda_{(\ell)}} ,
        c^{\beta}_{\lambda_{(1)} , \lambda_{(1)}} ,
        c^{\beta}_{\lambda , \lambda_{(1,\ell)}} ,
        c^{\beta}_{\lambda , \lambda_{(\ell,\uparrow)}} ,
        c^{\beta}_{\lambda , \lambda_{(1,\leftarrow)}}
        \Big)
        =
        (0 , 1 , 1 , 1 , 0 , 0) \, .
    \]
\end{lemma}

\begin{proof}[Proof of Theorem \ref{theorem:hook} ]
    By Lemma \ref{lemma:hook type 1} and equation (\ref{equation:main}),
    the only possibly negative coefficient in the Schur expansion of $(s_{\lambda}^{(1)})^2 - s_{\lambda}s_{\lambda}^{(2)}$ is that of $s_{(k^2 , 2^{\ell - 2})}$.
    However, by Lemma \ref{lemma:hook special} and equation (\ref{equation:main}), this coefficient is equal to
    \begin{align*}
        &
            - (n + k - 1)(n + 1 - \ell)
            + \frac{1}{2} (n + k - 1)(n + k)
            + \frac{1}{2} (n + 1 - \ell)(n + 2 - \ell)
        \\
        &
            + \frac{1}{2} (n + k - 1)(n + k - 2)
            + \frac{1}{2} (n + 1 - \ell)(n - \ell)
        \\
            = \,
        &
            - (n + k - 1)(n + 1 - \ell)
            + (n + k - 1)^2
            + (n + 1 - \ell)^2
        \\
            = \,
        &
            (n + k - 1)(k + \ell - 2)
            + (n + 1 - \ell)^2 \, ,
    \end{align*}
    which is positive.
\end{proof}

Before proving Lemma \ref{lemma:hook type 1},
we observe that it is obviously true when 
$c^{\alpha}_{\lambda , \lambda_{(1,\ell)}} = 0$, which occurs in the following case:

\begin{lemma} \label{lemma:hook LRT}
    Let $\alpha = (\alpha_i)_{i = 1}^{\infty}$ be a partition of $2|\lambda| - 2$ with $\alpha_3 > 2$ or $\alpha_{\ell + 1} > 1$.
    Then we have $c^{\alpha}_{\lambda , \lambda_{(1,\ell)}} = 0$.
\end{lemma}

\begin{proof}
    
    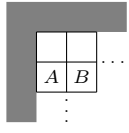
\begin{figure}
        \caption{A tableau of shape $\alpha/\lambda$, where $\alpha_3 > 2$ \vspace{5pt}}
        \label{figure:hook LRT}
        \begin{tikzpicture}[scale=0.4]
            \fill[gray] (0 , 0) -- (4 , 0) -- (4 , -1) -- (1 , -1) -- (1 , -4) -- (0 , -4) -- (0 , 0) ;
            
            \draw (1 , -1) -- (3 , -1) ;
            \draw (1 , -2) -- (3 , -2) ;
            \draw (1 , -3) -- (3 , -3) ;
            \draw (1 , -1) -- (1 , -3) ;
            \draw (2 , -1) -- (2 , -3) ;
            \draw (3 , -1) -- (3 , -3) ;
            \draw (3.6 , -2) node {\tiny $\cdots$} ;
            \draw (2 , -3.35) node {\tiny $\vdots$} ;
            \draw (1.5 , -2.5) node {\scriptsize $A$} ;
            \draw (2.5 , -2.5) node {\scriptsize $B$} ;
        \end{tikzpicture}
    \end{figure}
    
    If $\alpha_{\ell + 1} > 1$,
    then the second column of $\alpha/\lambda$ contains at least $\ell$ cells.
    However,
    we have $\ell(\lambda_{(1,\ell)}) = \ell - 1 < \ell$.
    Thus, we get $c^{\alpha}_{\lambda , \lambda_{(1,\ell)}} = 0$
    by Lemma \ref{lemma:badLRT}(a).
    
    Next,
    suppose $\alpha_3 > 2$ but $c^{\alpha}_{\lambda , \lambda_{(1 , \ell)}} > 0$.
    Let $T$ be a Littlewood-Richardson tableau of shape $\alpha/\lambda$ with content $\lambda_{(1 , \ell)}$.
    See Figure \ref{figure:hook LRT}.
    As $\alpha_3 > 2$,
    there are at least two cells in the third row of $T$.
    Let $A , B$ be the integers in the first two cells of the third row of $T$, so certainly $A \le B$.
    Since there is a cell above each of these two cells,
    both $A$ and $B$ are greater than $1$.
    Note that $\lambda_{(1 , \ell),i} \leq 1$ for all $i > 1$ as $\lambda$ is a hook, so we must have $A$ and $B$ are distinct, thus $A < B$.
    Then in the reverse reading word of $T$,
    the integer $B$ comes before the only $A$ in the word.
    This contradicts condition (b) of Definition \ref{definition:Littlewood-Richardson tableaux}.
    Therefore,
    there do not exist Littlewood-Richardson tableaux of shape $\alpha/\lambda$ with content $\lambda_{(1 , \ell)}$ in this case,
    and we must have $c^{\alpha}_{\lambda , \lambda_{(1,\ell)}} = 0$.
\end{proof}

\begin{proof}[Proof of Lemma \ref{lemma:hook type 1}]
    For each partition $\alpha$ of $2|\lambda| - 2$,
    containing $\lambda$ and not equal to $(k^2 , 2^{\ell - 2})$,
    in order to prove the inequality $c^{\alpha}_{\lambda_{(1)} , \lambda_{(\ell)}} \geq c^{\alpha}_{\lambda , \lambda_{(1,\ell)}}$,
    we construct an injection from $\LR(\alpha/\lambda , \lambda_{(1,\ell)})$,
    the set of Littlewood-Richardson tableaux of shape $\alpha/\lambda$ and content $\lambda_{(1,\ell)}$,
    to $\LR(\alpha/\lambda_{(\ell)} , \lambda_{(1)})$,
    the set of Littlewood-Richardson tableaux of shape $\alpha/\lambda_{(\ell)}$ and content $\lambda_{(1)}$.
    
    By Lemma \ref{lemma:hook LRT},
    it suffices to consider $\alpha = (\alpha_i)_{i = 1}^{\infty}$ with $\alpha_{\ell + 1} \leq 1$ and $\alpha_3 \leq 2$ (which implies $\alpha_{\ell} \leq 2$).
    We divide into $3$ cases, namely  (1)  $\alpha_{\ell} = 1$ (2)
    $\alpha_{\ell} = 2$ and $\alpha_{\ell + 1} = 1$ and (3) $\alpha_{\ell} = 2$ and $\alpha_{\ell + 1} = 0$.
    In each case,
    we give an algorithm that turns a Littlewood-Richardson tableau $T \in \LR(\alpha/\lambda , \lambda_{(1,\ell)})$
    into a Littlewood-Richardson tableau $T' \in \LR(\alpha/\lambda_{(\ell)} , \lambda_{(1)})$
    by adding an extra cell to the beginning of the $\ell$-th row and an extra integer $\ell$ to the content,
    and we check that $T'$ satisfies both conditions (a) and (b) in Definition \ref{definition:Littlewood-Richardson tableaux}.  

    As such we obtain a map
    $$\LR(\alpha/\lambda , \lambda_{(1,\ell)}) \to \LR(\alpha/\lambda_{(\ell)} , \lambda_{(1)})$$
    which by the construction is clearly injective, proving the result.
    
    \begin{enumerate}[label=$\bullet$]
        \begin{figure}
            \captionsetup{width=300pt}
            \caption{The map $\LR(\alpha/\lambda , \lambda_{(1,\ell)}) \to \LR(\alpha/\lambda_{(\ell)} , \lambda_{(1)})$ in Case 1. In this and the following figures that illustrate maps, the cells that will be changed are indicated by either a blue border, and the integer that is added is indicated by a red border \vspace{5pt}}
            \label{figure:hook case 1}
            \begin{tikzpicture}[scale=0.4]
                \draw
                    (-0.7 , -0.5) node {\tiny $1$}
                    (-0.7 , -1.5) node {\tiny $2$}
                    (-0.7 , -2.3) node {\tiny $\vdots$}
                    (-0.7 , -3.5) node {\tiny $\ell \!\! - \!\! 1$}
                    (-0.7 , -4.5) node {\tiny $\ell$}
                    (0.5 , 0.5) node {\tiny $1$}
                    (1.55 , 0.5) node {\tiny $\cdots$}
                    (2.55 , 0.5) node {\tiny $k$}
                ;
                
                \fill[gray] (0 , 0) -- (3 , 0) -- (3 , -1) -- (1 , -1) -- (1 , -5) -- (0 , -5) -- (0 , 0) ;
                \draw
                    (0 , -5) -- (1 , -5) -- (1 , -6) -- (0 , -6) -- (0 , -5)
                    ;
                \draw (0 , -7) -- (1 , -7) ;
                \draw (0 , -8) -- (1 , -8) ;
                \draw (0 , -7) -- (0 , -8) ;
                \draw (1 , -7) -- (1 , -8) ;
                \draw (0.5 , -6.3) node {\tiny $\vdots$} ;
                \draw (1.55 , -1.5) node {\tiny $\cdots$} ;
                \draw (3.55 , -0.5) node {\tiny $\cdots$} ;
                \draw[blue , very thick] (0 , -5) -- (1 , -5) -- (1 , -8) -- (0 , -8) -- (0 , -5) ;
                
                \draw (5.3 , -4.5) node {$\mapsto$} ;
                
                \draw
                    (7.3 , -0.5) node {\tiny $1$}
                    (7.3 , -1.5) node {\tiny $2$}
                    (7.3 , -2.3) node {\tiny $\vdots$}
                    (7.3 , -3.5) node {\tiny $\ell \!\! - \!\! 1$}
                    (7.3 , -4.5) node {\tiny $\ell$}
                    
                    (8.5 , 0.5) node {\tiny $1$}
                    (9.55 , 0.5) node {\tiny $\cdots$}
                    (10.55 , 0.5) node {\tiny $k$}
                ;
                
                \fill[gray]
                    (8 , 0) -- (11 , 0) -- (11 , -1) -- (9 , -1) -- (9 , -4) -- (8 , -4) -- (8 , 0)
                ;
                \draw (8 , -4) -- (9 , -4) ;
                \draw (8 , -5) -- (9 , -5) ;
                \draw (8 , -4) -- (8 , -5) ;
                \draw (9 , -4) -- (9 , -5) ;
                \draw (8 , -6) -- (9 , -6) ;
                \draw (8 , -7) -- (9 , -7) ;
                \draw (8 , -8) -- (9 , -8) ;
                \draw (8 , -6) -- (8 , -8) ;
                \draw (9 , -6) -- (9 , -8) ;
                \draw (8.5 , -5.35) node {\tiny $\vdots$} ;
                \draw (9.55 , -1.5) node {\tiny $\cdots$} ;
                \draw (11.55 , -0.5) node {\tiny $\cdots$} ;
                \draw[blue , very thick] (8 , -4) -- (9 , -4) -- (9 , -6.95) -- (8 , -6.95) -- (8 , -4) ;
                \draw[red , very thick] (8 , -7.05) -- (9 , -7.05) -- (9 , -8) -- (8 , -8) -- (8 , -7.05) ;
                \draw (8.5 , -7.5) node {\scriptsize $\ell$} ;
            \end{tikzpicture}
        \end{figure}
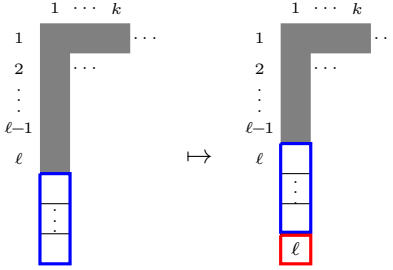
        
        \item \textbf{Case 1: } $\alpha_{\ell} = 1$.
        \begin{enumerate}[label=$\circ$]
            \item $T'$ is obtained by moving the first column of $T$ one cell upward,
            and then placing the integer $\ell$ in the last cell.
            
            \item Then (a) is
            clear as $\ell$ is larger than all other integers in $T$. For (b),
            the integer $\ell$ is added to the end of the reverse reading word of $T$,
            so the result remains a lattice permutation.
        \end{enumerate}
        
        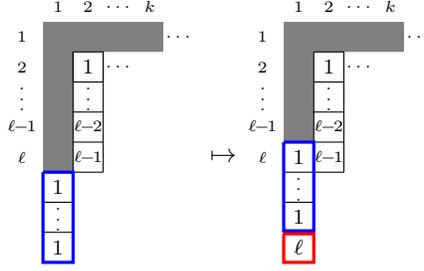
\begin{figure}
            \caption{The map $\LR(\alpha/\lambda , \lambda_{(1,\ell)}) \to \LR(\alpha/\lambda_{(\ell)} , \lambda_{(1)})$ in Case 2 \vspace{5    pt}}
            \label{figure:hook case 2}
            \begin{tikzpicture}[scale=0.4]
                \draw
                    (-0.7 , -0.5) node {\tiny $1$}
                    (-0.7 , -1.5) node {\tiny $2$}
                    (-0.7 , -2.3) node {\tiny $\vdots$}
                    (-0.7 , -3.5) node {\tiny $\ell \!\! - \!\! 1$}
                    (-0.7 , -4.5) node {\tiny $\ell$}
                    (0.5 , 0.5) node {\tiny $1$}
                    (1.5 , 0.5) node {\tiny $2$}
                    (2.55 , 0.5) node {\tiny $\cdots$}
                    (3.55 , 0.5) node {\tiny $k$}
                ;
                
                \fill[gray] (0 , 0) -- (4 , 0) -- (4 , -1) -- (1 , -1) -- (1 , -5) -- (0 , -5) -- (0 , 0) ;
                \draw (1 , -1) -- (2 , -1) -- (2 , -2) -- (1 , -2) -- (1 , -1) ;
                \draw (1 , -3) -- (2 , -3) ;
                \draw (1 , -4) -- (2 , -4) ;
                \draw (1 , -5) -- (2 , -5) ;
                \draw (1 , -1) -- (1 , -5) ;
                \draw (2 , -1) -- (2 , -5) ;
                \draw (4.55 , -0.5) node {\tiny $\cdots$} ;
                \draw (2.55 , -1.5) node {\tiny $\cdots$} ;
                \draw (1.5 , -2.3) node {\tiny $\vdots$} ;
                \draw (1.5 , -1.5) node {\footnotesize $1$} ;
                \draw (1.5 , -3.5) node {\tiny $\ell \!\! - \!\! 2$} ;
                \draw (1.5 , -4.5) node {\tiny $\ell \!\! - \!\! 1$} ;
                \draw (0 , -5) -- (1 , -5) -- (1 , -6) -- (0 , -6) -- (0 , -5) ;
                \draw (0 , -7) -- (1 , -7) ;
                \draw (0 , -8) -- (1 , -8) ;
                \draw (0 , -7) -- (0 , -8) ;
                \draw (1 , -7) -- (1 , -8) ;
                \draw (0.5 , -6.3) node {\tiny $\vdots$} ;
                \draw (0.5 , -5.5) node {\footnotesize $1$} ;
                \draw (0.5 , -7.5) node {\footnotesize $1$} ;
                \draw[blue , very thick] (0 , -5) -- (1 , -5) -- (1 , -8) -- (0 , -8) -- (0 , -5) ;
                
                \draw (6 , -4.5) node {$\mapsto$} ;
                
                \draw
                    (7.3 , -0.5) node {\tiny $1$}
                    (7.3 , -1.5) node {\tiny $2$}
                    (7.3 , -2.3) node {\tiny $\vdots$}
                    (7.3 , -3.5) node {\tiny $\ell \!\! - \!\! 1$}
                    (7.3 , -4.5) node {\tiny $\ell$}
                    (8.5 , 0.5) node {\tiny $1$}
                    (9.5 , 0.5) node {\tiny $2$}
                    (10.55 , 0.5) node {\tiny $\cdots$}
                    (11.55 , 0.5) node {\tiny $k$}
                ;
                
                \fill[gray] (8 , 0) -- (12 , 0) -- (12 , -1) -- (9 , -1) -- (9 , -4) -- (8 , -4) -- (8 , 0) ;
                
                \draw
                    (9 , -1) -- (10 , -1)
                    (9 , -2) -- (10 , -2)
                    (9 , -3) -- (10 , -3)
                    (9 , -4) -- (10 , -4)
                    (9 , -5) -- (10 , -5)
                    (9 , -1) -- (9 , -5)
                    (10 , -1) -- (10 , -5)
                ;
                \draw (12.55 , -0.5) node {\tiny $\cdots$} ;
                \draw (10.55 , -1.5) node {\tiny $\cdots$} ;
                \draw (9.5 , -2.3) node {\tiny $\vdots$} ;
                \draw (9.5 , -1.5) node {\footnotesize $1$} ;
                \draw (9.5 , -3.5) node {\tiny $\ell \!\! - \!\! 2$} ;
                \draw (9.5 , -4.5) node {\tiny $\ell \!\! - \!\! 1$} ;
                \draw (8 , -4) -- (9 , -4) ;
                \draw (8 , -5) -- (9 , -5) ;
                \draw (8 , -4) -- (8 , -5) ;
                \draw (9 , -4) -- (9 , -5) ;
                \draw (8 , -6) -- (9 , -6) ;
                \draw (8 , -7) -- (9 , -7) ;
                \draw (8 , -8) -- (9 , -8) ;
                \draw (8 , -6) -- (8 , -8) ;
                \draw (9 , -6) -- (9 , -8) ;
                \draw (8.5 , -5.3) node {\tiny $\vdots$} ;
                \draw (8.5 , -4.5) node {\footnotesize $1$} ;
                \draw (8.5 , -6.5) node {\footnotesize $1$} ;
                \draw[blue , very thick] (8 , -4) -- (9 , -4) -- (9 , -6.95) -- (8 , -6.95) -- (8 , -4) ;
                \draw[red , very thick] (8 , -7.05) -- (9 , -7.05) -- (9 , -8) -- (8 , -8) -- (8 , -7.05) ;
                \draw (8.5 , -7.5) node {\footnotesize $\ell$} ;
            \end{tikzpicture}
        \end{figure}
        
        \item \textbf{Case 2: } $\alpha_{\ell} = 2$ and $\alpha_{\ell + 1} = 1$.
        \begin{enumerate}[label=$\circ$]
            \item $T'$ is obtained by
            moving the first column of $T$ one cell upward,
            and then placing the integer $\ell$ in the last cell.
            
            \item Then (a) is clear and (b) is proved similar to Case 1.
        \end{enumerate}
        
        \begin{figure}
            \caption{The map $\LR(\alpha/\lambda , \lambda_{(1,\ell)}) \to \LR(\alpha/\lambda_{(\ell)} , \lambda_{(1)})$ in Case 3\vspace{5pt}}
            \label{figure:hook case 3}
            \begin{tikzpicture}[scale=0.4]
                \draw
                    (-0.7 , -0.5) node {\tiny $1$}
                    (-0.7 , -1.5) node {\tiny $2$}
                    (-0.7 , -2.5) node {\tiny $3$}
                    (-0.7 , -3.5) node {\tiny $4$}
                    (-0.7 , -4.3) node {\tiny $\vdots$}
                    (-0.7 , -5.5) node {\tiny $\ell$}
                    (0.5 , 0.5) node {\tiny $1$}
                    (1.5 , 0.5) node {\tiny $2$}
                    (2.55 , 0.5) node {\tiny $\cdots$}
                    (3.5 , 0.5) node {\tiny $\alpha_2$}
                    (4.55 , 0.5) node {\tiny $\cdots$}
                    (5.55 , 0.5) node {\tiny $k$}
                ;

                \fill[gray] (0 , 0) -- (6 , 0) -- (6 , -1) -- (1 , -1) -- (1 , -6) -- (0 , -6) -- (0 , 0) ;
                \draw (1 , -1) -- (4 , -1) ;
                \draw (1 , -2) -- (4 , -2) ;
                \draw (1 , -3) -- (2 , -3) ;
                \draw (1 , -4) -- (2 , -4) ;
                \draw (1 , -1) -- (1 , -4) ;
                \draw (2 , -1) -- (2 , -4) ;

                \draw (1 , -5) -- (2 , -5) ;
                \draw (1 , -6) -- (2 , -6) ;
                \draw (1 , -5) -- (1 , -6) ;
                \draw (2 , -5) -- (2 , -6) ;

                \draw (3 , -1) -- (4 , -1) ;
                \draw (3 , -2) -- (4 , -2) ;
                \draw (3 , -1) -- (3 , -2) ;
                \draw (4 , -1) -- (4 , -2) ;

                \draw (6 , 0) -- (7 , 0) ;
                \draw (6 , -1) -- (7 , -1) ;
                \draw (6 , 0) -- (6 , -1) ;
                \draw (7 , 0) -- (7 , -1) ;

                \draw (2.55 , -1.5) node {\tiny $\cdots$} ;
                \draw (7.55 , -0.5) node {\tiny $\cdots$} ;
                \draw (1.5 , -4.3) node {\tiny $\vdots$} ;
                \draw (1.5 , -1.5) node {\footnotesize $1$} ;
                \draw (1.5 , -2.5) node {\footnotesize $2$} ;
                \draw (1.5 , -3.5) node {\footnotesize $3$} ;
                \draw (1.5 , -5.5) node {\tiny $\ell \!\! - \!\! 1$} ;
                \draw (3.5 , -1.5) node {\footnotesize $1$} ;
                \draw (6.5 , -0.5) node {\footnotesize $1$} ;

                \draw[blue , very thick]
                    (1 , -2) -- (2 , -2) -- (2 , -3) -- (1 , -3) -- (1 , -2)
                    (1 , -3) -- (2 , -3) -- (2 , -6) -- (1 , -6) -- (1 , -3)
                    (3 , -1) -- (4 , -1) -- (4 , -2) -- (3 , -2) -- (3 , -1)
                    ;

                \draw (9 , -3) node {$\mapsto$} ;

                \draw
                    (10.3 , -0.5) node {\tiny $1$}
                    (10.3 , -1.5) node {\tiny $2$}
                    (10.3 , -2.5) node {\tiny $3$}
                    (10.3 , -3.3) node {\tiny $\vdots$}
                    (10.3 , -4.5) node {\tiny $\ell \!\! - \!\! 1$}
                    (10.3 , -5.5) node {\tiny $\ell$}
                    (11.5 , 0.5) node {\tiny $1$}
                    (12.5 , 0.5) node {\tiny $2$}
                    (13.55 , 0.5) node {\tiny $\cdots$}
                    (14.5 , 0.5) node {\tiny $\alpha_2$}
                    (15.55 , 0.5) node {\tiny $\cdots$}
                    (16.55 , 0.5) node {\tiny $k$}
                ;

                \fill[gray]
                    (11 , 0) -- (17 , 0) -- (17 , -1) -- (12 , -1) -- (12 , -5) -- (11 , -5) -- (11 , 0)
                ;

                \draw
                    (12 , -1) -- (15 , -1)
                    (12 , -2) -- (15 , -2)
                    (12 , -3) -- (13 , -3)
                    (12 , -1) -- (12 , -3)
                    (13 , -1) -- (13 , -3)

                    (12 , -4) -- (13 , -4)
                    (11 , -5) -- (13 , -5)
                    (11 , -6) -- (13 , -6)
                    (11 , -5) -- (11 , -6)
                    (12 , -4) -- (12 , -6)
                    (13 , -4) -- (13 , -6)

                    (14 , -1) -- (15 , -1)
                    (14 , -2) -- (15 , -2)
                    (14 , -1) -- (14 , -2)
                    (15 , -1) -- (15 , -2)

                    (17 , 0) -- (18 , 0)
                    (17 , -1) -- (18 , -1)
                    (17 , 0) -- (17 , -1)
                    (18 , 0) -- (18 , -1)
                ;

                \draw
                    (13.55 , -1.5) node {\tiny $\cdots$}
                    (18.55 , -0.5) node {\tiny $\cdots$}
                    (12.5 , -3.3) node {\tiny $\vdots$}

                    (12.5 , -1.5) node {\footnotesize $1$}
                    (12.5 , -2.5) node {\footnotesize $3$}
                    (12.5 , -4.5) node {\tiny $\ell \!\! - \!\! 1$}
                    (14.5 , -1.5) node {\footnotesize $2$}
                    (17.5 , -0.5) node {\footnotesize $1$}
                    (11.5 , -5.5) node {\footnotesize $1$}
                ;

                \draw
                    (14.5 , -1.5) node {\footnotesize $2$}
                    (11.5 , -5.5) node {\footnotesize $1$}
                ;
                
                \draw[blue , very thick]
                    (12 , -2) -- (13 , -2) -- (13 , -4.95) -- (12 , -4.95) -- (12 , -2)
                    (11 , -5) -- (11.95 , -5) -- (11.95 , -6) -- (11 , -6) -- (11 , -5)
                    (14 , -1) -- (15 , -1) -- (15 , -2) -- (14 , -2) -- (14 , -1)
                    ;
                \draw[red , very thick] (12.05 , -5.05) -- (13 , -5.05) -- (13 , -6) -- (12.05 , -6) -- (12.05 , -5.05) ;
                \draw (12.5 , -5.5) node {\footnotesize $\ell$} ;
            \end{tikzpicture}
        \end{figure}
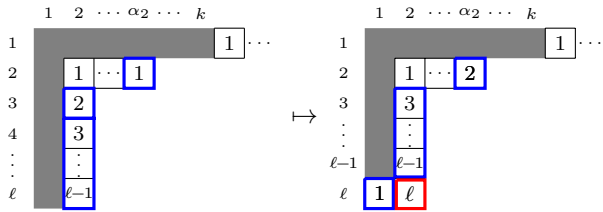
        
        \item \textbf{Case 3: } $\alpha_{\ell} = 2$ and $\alpha_{\ell + 1} = 0$, but $\alpha \neq (k^2 , 2^{\ell - 2})$.
        \begin{enumerate}[label=$\circ$]
            \item 
            Note first the first column of $T$ must contain $1 , 2 , \ldots , \ell - 1$ since it is strictly increasing.
            We define $T'$ by moving the last $1$ in the second row to the first cell of the $\ell$-th row,
            and move the $2$ to the end of the second row,
            move the $3 , \ldots , \ell - 1$ one cell upward,
            and place the integer $\ell$ in the last cell of the second column.
            
            \item Then (a) is clear.  For (b), 
            note that there must be $1$'s in the first row
            as $\alpha \neq (k^2 , 2^{\ell - 2})$.
            Thus, there is at least one $1$ in front of the $2$ in the reverse word of $T'$.
        \end{enumerate}
    \end{enumerate}    
\end{proof}

\begin{proof}[Proof of Lemma \ref{lemma:hook special}]
    
    \begin{figure}
        \caption{Skew shapes $\beta / \lambda$ and $\beta / \lambda_{(\ell)}$}
        \label{figure:hook skew shapes}
        \begin{subfigure}{0.4\textwidth}
            \subcaption{$\beta / \lambda = (k^2 , 2^{\ell - 2}) / (k , 1^{\ell - 1})$ \vspace{5pt}}
            \centering \begin{tikzpicture}[scale=0.4]
                \draw
                    (-0.7 , -0.5) node {\tiny $1$}
                    (-0.7 , -1.5) node {\tiny $2$}
                    (-0.7 , -2.5) node {\tiny $3$}
                    (-0.7 , -3.3) node {\tiny $\vdots$}
                    (-0.7 , -4.5) node {\tiny $\ell \!\! - \!\! 1$}
                    (-0.7 , -5.5) node {\scriptsize $\ell$}
                    (0.5 , 0.5) node {\tiny $1$}
                    (1.5 , 0.5) node {\tiny $2$}
                    (2.5 , 0.5) node {\tiny $3$}
                    (3.55 , 0.5) node {\tiny $\cdots$}
                    (4.55 , 0.5) node {\tiny $k \!\! - \!\! 1$}
                    (5.55 , 0.5) node {\tiny $k$}
                ;

                \fill[gray]
                    (0 , 0) -- (6 , 0) -- (6 , -1) -- (1 , -1) -- (1 , -6) -- (0 , -6) -- (0 , 0)
                ;
                
                \draw[gray]
                    (3.55 , -0.5) node {\tiny $\cdots$}
                    (0.5 , -3.3) node {\tiny $\vdots$}
                ;
                
                \draw
                    (1 , -1) -- (6 , -1)
                    (1 , -2) -- (6 , -2)
                    (1 , -3) -- (2 , -3)
                    (1 , -4) -- (2 , -4)
                    (1 , -5) -- (2 , -5)
                    (1 , -6) -- (2 , -6)
                    (1 , -1) -- (1 , -6)
                    (2 , -1) -- (2 , -6)
                    (3 , -1) -- (3 , -2)
                    (4 , -1) -- (4 , -2)
                    (5 , -1) -- (5 , -2)
                    (6 , -1) -- (6 , -2)
                ;
                
                \draw
                    (3.55 , -1.5) node {\tiny $\cdots$}
                    (1.5 , -3.3) node {\tiny $\vdots$}
                ;
            \end{tikzpicture}
        \end{subfigure}
        \begin{subfigure}{0.4\textwidth}
            \caption{$\beta / \lambda_{(\ell)} = (k^2 , 2^{\ell - 2}) / (k , 1^{\ell - 2})$ \vspace{5pt}}
            \centering \begin{tikzpicture}[scale=0.4]
                \draw
                    (-0.7 , -0.5) node {\tiny $1$}
                    (-0.7 , -1.5) node {\tiny $2$}
                    (-0.7 , -2.5) node {\tiny $3$}
                    (-0.7 , -3.3) node {\tiny $\vdots$}
                    (-0.7 , -4.5) node {\tiny $\ell \!\! - \!\! 1$}
                    (-0.7 , -5.5) node {\tiny $\ell$}
                    (0.5 , 0.5) node {\tiny $1$}
                    (1.5 , 0.5) node {\tiny $2$}
                    (2.5 , 0.5) node {\tiny $3$}
                    (3.55 , 0.5) node {\tiny $\cdots$}
                    (4.55 , 0.5) node {\tiny $k \!\! - \!\! 1$}
                    (5.55 , 0.5) node {\tiny $k$}
                ;

                \fill[gray]
                    (0 , 0) -- (6 , 0) -- (6 , -1) -- (1 , -1) -- (1 , -5) -- (0 , -5) -- (0 , 0)
                ;
                
                \draw
                    (1 , -1) -- (6 , -1)
                    (1 , -2) -- (6 , -2)
                    (1 , -3) -- (2 , -3)
                    (1 , -4) -- (2 , -4)
                    (0 , -5) -- (2 , -5)
                    (0 , -6) -- (2 , -6)
                    (0 , -5) -- (0 , -6)
                    (1 , -1) -- (1 , -6)
                    (2 , -1) -- (2 , -6)
                    (3 , -1) -- (3 , -2)
                    (4 , -1) -- (4 , -2)
                    (5 , -1) -- (5 , -2)
                    (6 , -1) -- (6 , -2)
                ;

                \draw
                    (3.55 , -1.5) node {\tiny $\cdots$}
                    (1.5 , -3.3) node {\tiny $\vdots$}
                ;
            \end{tikzpicture}
        \end{subfigure}
    \end{figure}
    
    We claim $c^{\beta}_{\lambda_{(1)} , \lambda_{(\ell)}} = c^{\beta}_{\lambda , \lambda_{(1,\leftarrow)}} = 0$
    by Lemma $\ref{lemma:badLRT}$(c).  In fact,  both skew shapes $\beta / \lambda_{(\ell)}$ and $\beta / \lambda$ have only $\ell - 1$ non-empty rows (see Figure \ref{figure:hook skew shapes})
    but $\ell(\lambda_{(1)}) = \ell(\lambda_{(1,\leftarrow)}) = \ell$,
    so there do not exist
    Littlewood-Richardson tableaux of shape $\beta/\lambda_{(\ell)}$ and content $\lambda_{(1)}$ or
    those of shape $\beta/\lambda$ and content $\lambda_{(1,\leftarrow)}$.
    
    Next we claim $c^{\beta}_{\lambda , \lambda_{(\ell,\uparrow)}} = 0$, since
    by Lemma $\ref{lemma:badLRT}$(a):
    the second column of $\beta / \lambda_{(\ell)}$ has $\ell - 1$ cells (see Figure \ref{figure:hook skew shapes})
    but $\ell(\lambda_{(\ell,\uparrow)}) = \ell - 2$,
    so there do not exist
    Littlewood-Richardson tableaux of shape $\beta/\lambda_{(\ell)}$ and content $\lambda_{(\ell,\uparrow)}$.
    
    \begin{figure}
        \captionsetup{width=280pt}
        \caption{Littlewood-Richardson tableaux in $\LR(\beta / \lambda_{(\ell)} , \lambda_{(\ell)})$, $\LR(\beta / \lambda_{(1)} , \lambda_{(1)})$, $\LR(\beta / \lambda , \lambda_{(1 , \ell)})$}
        \label{figure:hook special}
        \begin{subfigure}{0.3\textwidth}
            \subcaption{$T \in \LR(\beta / \lambda_{(\ell)} , \lambda_{(\ell)})$ \vspace{5pt}}
            \centering \begin{tikzpicture}[scale=0.4]
                \draw
                    (-0.7 , -0.5) node {\tiny $1$}
                    (-0.7 , -1.5) node {\tiny $2$}
                    (-0.7 , -2.5) node {\tiny $3$}
                    (-0.7 , -3.3) node {\tiny $\vdots$}
                    (-0.7 , -4.5) node {\tiny $\ell \!\! - \!\! 1$}
                    (-0.7 , -5.5) node {\tiny $\ell$}
                    (0.5 , 0.5) node {\tiny $1$}
                    (1.5 , 0.5) node {\tiny $2$}
                    (2.5 , 0.5) node {\tiny $3$}
                    (3.55 , 0.5) node {\tiny $\cdots$}
                    (4.55 , 0.5) node {\tiny $k \!\! - \!\! 1$}
                    (5.55 , 0.5) node {\tiny $k$}
                ;

                \fill[gray]
                    (0 , 0) -- (6 , 0) -- (6 , -1) -- (1 , -1) -- (1 , -5) -- (0 , -5) -- (0 , 0)
                ;

                \draw
                    (1 , -1) -- (6 , -1)
                    (1 , -2) -- (6 , -2)
                    (1 , -3) -- (2 , -3)
                    (1 , -4) -- (2 , -4)
                    (0 , -5) -- (2 , -5)
                    (0 , -6) -- (2 , -6)
                    (0 , -5) -- (0 , -6)
                    (1 , -1) -- (1 , -6)
                    (2 , -1) -- (2 , -6)
                    (3 , -1) -- (3 , -2)
                    (4 , -1) -- (4 , -2)
                    (5 , -1) -- (5 , -2)
                    (6 , -1) -- (6 , -2)
                ;

                \draw
                    (3.55 , -1.5) node {\tiny $\cdots$}
                    (1.5 , -3.3) node {\tiny $\vdots$}
                    (1.5 , -1.5) node {\footnotesize $1$}
                    (2.5 , -1.5) node {\footnotesize $1$}
                    (4.5 , -1.5) node {\footnotesize $1$}
                    (5.5 , -1.5) node {\footnotesize $1$}
                    (1.5 , -2.5) node {\footnotesize $2$}
                    (1.5 , -4.5) node {\tiny $\ell \!\! - \!\! 2$}
                    (1.5 , -5.5) node {\tiny $\ell \!\! - \!\! 1$}
                    (0.5 , -5.5) node {\footnotesize $1$}
                ;
            \end{tikzpicture}
        \end{subfigure}
        \begin{subfigure}{0.3\textwidth}
            \subcaption{$T \in \LR(\beta / \lambda_{(1)} , \lambda_{(1)})$ \vspace{5pt}}
            \centering \begin{tikzpicture}[scale=0.4]
                \draw
                    (-0.7 , -0.5) node {\tiny $1$}
                    (-0.7 , -1.5) node {\tiny $2$}
                    (-0.7 , -2.5) node {\tiny $3$}
                    (-0.7 , -3.3) node {\tiny $\vdots$}
                    (-0.7 , -4.5) node {\tiny $\ell \!\! - \!\! 1$}
                    (-0.7 , -5.5) node {\tiny $\ell$}
                    (0.5 , 0.5) node {\tiny $1$}
                    (1.5 , 0.5) node {\tiny $2$}
                    (2.5 , 0.5) node {\tiny $3$}
                    (3.55 , 0.5) node {\tiny $\cdots$}
                    (4.55 , 0.5) node {\tiny $k \!\! - \!\! 1$}
                    (5.55 , 0.5) node {\tiny $k$}
                ;

                \fill[gray]
                    (0 , 0) -- (5 , 0) -- (5 , -1) -- (1 , -1) -- (1 , -6) -- (0 , -6) -- (0 , 0)
                ;

                \draw
                    (5 , 0) -- (6 , 0)
                    (1 , -1) -- (6 , -1)
                    (1 , -2) -- (6 , -2)
                    (1 , -3) -- (2 , -3)
                    (1 , -4) -- (2 , -4)
                    (1 , -5) -- (2 , -5)
                    (1 , -6) -- (2 , -6)
                    (1 , -1) -- (1 , -6)
                    (2 , -1) -- (2 , -6)
                    (3 , -1) -- (3 , -2)
                    (4 , -1) -- (4 , -2)
                    (5 , 0) -- (5 , -2)
                    (6 , 0) -- (6 , -2)
                ;
                
                \draw
                    (3.55 , -1.5) node {\tiny $\cdots$}
                    (1.5 , -3.3) node {\tiny $\vdots$}
                    (1.5 , -1.5) node {\footnotesize $1$}
                    (2.5 , -1.5) node {\footnotesize $1$}
                    (4.5 , -1.5) node {\footnotesize $1$}
                    (5.5 , -0.5) node {\footnotesize $1$}
                    (5.5 , -1.5) node {\footnotesize $2$}
                    (1.5 , -2.5) node {\footnotesize $3$}
                    (1.5 , -4.5) node {\tiny $\ell \!\! - \!\! 1$}
                    (1.5 , -5.5) node {\footnotesize $\ell$}
                ;
            \end{tikzpicture}
        \end{subfigure}
        \begin{subfigure}{0.3\textwidth}
            \subcaption{$T \in \LR(\beta / \lambda , \lambda_{(1 , \ell)})$ \vspace{5pt}}
            \begin{tikzpicture}[scale=0.4]
                \draw
                    (-0.7 , -0.5) node {\tiny $1$}
                    (-0.7 , -1.5) node {\tiny $2$}
                    (-0.7 , -2.5) node {\tiny $3$}
                    (-0.7 , -3.3) node {\tiny $\vdots$}
                    (-0.7 , -4.5) node {\tiny $\ell \!\! - \!\! 1$}
                    (-0.7 , -5.5) node {\tiny $\ell$}
                    (0.5 , 0.5) node {\tiny $1$}
                    (1.5 , 0.5) node {\tiny $2$}
                    (2.5 , 0.5) node {\tiny $3$}
                    (3.55 , 0.5) node {\tiny $\cdots$}
                    (4.55 , 0.5) node {\tiny $k \!\! - \!\! 1$}
                    (5.55 , 0.5) node {\tiny $k$}
                ;

                \fill[gray]
                    (0 , 0) -- (6 , 0) -- (6 , -1) -- (1 , -1) -- (1 , -6) -- (0 , -6) -- (0 , 0)
                ;

                \draw
                    (1 , -1) -- (6 , -1)
                    (1 , -2) -- (6 , -2)
                    (1 , -3) -- (2 , -3)
                    (1 , -4) -- (2 , -4)
                    (1 , -5) -- (2 , -5)
                    (1 , -6) -- (2 , -6)
                    (1 , -1) -- (1 , -6)
                    (2 , -1) -- (2 , -6)
                    (3 , -1) -- (3 , -2)
                    (4 , -1) -- (4 , -2)
                    (5 , -1) -- (5 , -2)
                    (6 , -1) -- (6 , -2)
                ;

                \draw
                    (3.55 , -1.5) node {\tiny $\cdots$}
                    (1.5 , -3.3) node {\tiny $\vdots$}
                    (1.5 , -1.5) node {\footnotesize $1$}
                    (2.5 , -1.5) node {\footnotesize $1$}
                    (4.5 , -1.5) node {\footnotesize $1$}
                    (5.5 , -1.5) node {\footnotesize $1$}
                    (1.5 , -2.5) node {\footnotesize $2$}
                    (1.5 , -4.5) node {\tiny $\ell \!\! - \!\! 2$}
                    (1.5 , -5.5) node {\tiny $\ell \!\! - \!\! 1$}
                ;
            \end{tikzpicture}
        \end{subfigure}
    \end{figure}
    
    Finally, one may check that each of
    $\LR(\beta / \lambda_{(\ell)}$ , $\lambda_{(\ell)})$ ,
    $\LR(\beta / \lambda_{(1)} , \lambda_{(1)})$ and
    $\LR(\beta / \lambda , \lambda_{(1 , \ell)})$
    is a singleton,
    the only element of which is depicted in Figure \ref{figure:hook special}.
    Therefore, we have $c^{\beta}_{\lambda_{(\ell)} , \lambda_{(\ell)}} =
    c^{\beta}_{\lambda_{(1)} , \lambda_{(1)}} =
    c^{\beta}_{\lambda , \lambda_{(1,\ell)}} = 1$ which completes the proof.
\end{proof}

Finally, we deal with hooks $\lambda = (k , 1^{\ell - 1})$ with $k  = 2$ or $\ell = 2$ separately.

\begin{theorem} \label{theorem:hook2}
    Let $\lambda$ be a hook,
    i.e.\ a partition of the form $(k , 1^{\ell - 1})$, where $k  = 2$ or $\ell = 2$.
    Then $(s_{\lambda}^{(1)})^2 - s_{\lambda}s_{\lambda}^{(2)}$ is Schur positive.
\end{theorem}

\begin{proof}
    We will prove this theorem by direct computation.
    The main fact that we will use in the computation is that the product of $s_{\lambda}$ and $s_{(1^\ell)}$ equals the sum of Schur polynomials $s_{\mu}$ corresponding to all partitions $\mu$ that can be obtained by adding one block each to $\ell$ rows of $\lambda$.
    \begin{itemize}
        \item
        \textbf{Case 1:} $\lambda = (2 , 1)$. \\
        By (\ref{equation:derived Schur}), we have
        \begin{align*}
            & s^{(1)}_{\lambda} = (n + 1) s_{(1,1)} + (n - 1) s_{2} \\
            & s^{(2)}_{\lambda} = (n + 1) (n - 1) s_{1} \, .
        \end{align*}
        Thus,
        \begin{align*}
            & (s_{\lambda}^{(1)})^2 - s_{\lambda}s_{\lambda}^{(2)} \\
            = \, & (n + 1)^2 (s_{(2,2)} + s_{(2,1,1)} + s_{(1,1,1,1)}) \\
            & + 2 (n + 1) (n - 1) (s_{(3,1)} + s_{(2,1,1)}) \\
            & + (n - 1)^2 (s_{(2,2)} + s_{(3,1)} + s_{(4)}) \\
            & - (n + 1) (n - 1) (s_{(3,1)} + s_{(2,2)} + s_{(2,1,1)}) \, .
        \end{align*}
        It is easy to check that this polynomial is Schur positive.

        \item
        \textbf{Case 2:} $\lambda = (2 , 1^{\ell - 1})$ with $\ell \geq 3$. \\
        By (\ref{equation:derived Schur}), we have
        \begin{align*}
            & s^{(1)}_{\lambda} = (n + 1) s_{(1^{\ell})} + (n + 1 - \ell) s_{(2,1^{\ell-2})} \\
            & s^{(2)}_{\lambda} = (n + 1) (n + 1 - \ell) s_{(1^{\ell - 1})} + \frac{1}{2} (n + 1 - \ell) (n + 2 - \ell) s_{(2,1^{\ell-3})} \, .
        \end{align*}
        Thus,
        \begin{align*}
            & (s_{\lambda}^{(1)})^2 - s_{\lambda}s_{\lambda}^{(2)} = \\
            & (n + 1)^2 s_{(1^{\ell})}^2 + 2 (n + 1) (n + 1 - \ell) s_{(1^{\ell})} s_{(2,1^{\ell-2})} + (n + 1 - \ell)^2 s_{(2,1^{\ell-2})}^2 \\
            & - (n + 1) (n + 1 - \ell) s_{\lambda} s_{(1^{\ell - 1})} - \frac{1}{2} (n + 1 - \ell) (n + 2 - \ell) s_{\lambda} s_{(2,1^{\ell-3})} \, .
        \end{align*}
        By Lemma \ref{lemma:type 2&3}, we have that $s_{(2,1^{\ell-2})}^2 - s_{\lambda} s_{(2,1^{\ell-3})}$ is Schur positive, which implies that
        \[
            (n + 1 - \ell)^2 s_{(2,1^{\ell-2})}^2 - \frac{1}{2} (n + 1 - \ell) (n + 2 - \ell) s_{\lambda} s_{(2,1^{\ell-3})}
        \]
        is Schur positive.
        The sum of the other three terms in $(s_{\lambda}^{(1)})^2 - s_{\lambda}s_{\lambda}^{(2)}$ is also Schur-positive by the following.
        \begin{align*}
            s_{(1^{\ell})}^2 &
            = s_{(2^{\ell})} + s_{(2^{\ell-1},1^2)} + \cdots + s_{(2,1^{2\ell-2})} + s_{(1^{2\ell})} \\
            s_{(1^{\ell})} s_{(2,1^{\ell-2})} &
            = s_{(3,2^{\ell-2},1)} + s_{(3,2^{\ell-3},1^3)} + \cdots + s_{(3,1^{2\ell-3})} \\ &
            \quad + s_{(2^{\ell-1},1^2)} + \cdots + s_{(2,1^{2\ell-2})} \\
            s_{\lambda} s_{(1^{\ell - 1})} &
            = s_{(3,2^{\ell-2},1)} + s_{(3,2^{\ell-3},1^3)} + \cdots + s_{(3,1^{2\ell-3})} \\ &
            \quad + s_{(2^{\ell})} + s_{(2^{\ell-1},1,1)} + \cdots + s_{(2,1^{2\ell-2})}
        \end{align*}
        Therefore, $(s_{\lambda}^{(1)})^2 - s_{\lambda}s_{\lambda}^{(2)}$ is Schur positive.

        \item
        \textbf{Case 3:} $\lambda = (k , 1)$ with $k \geq 3$. \\
        By (\ref{equation:derived Schur}), we have
        \begin{align*}
            & s^{(1)}_{\lambda} = (n - 1) s_{(k)} + (n + k - 1) s_{(k-1,1)} \\
            & s^{(2)}_{\lambda} = (n - 1) (n + k - 1) s_{(k-1)} + \frac{1}{2} (n + k - 1) (n + k - 2) s_{(k-2,1)} \, .
        \end{align*}
        Thus,
        \begin{align*}
            & (s_{\lambda}^{(1)})^2 - s_{\lambda}s_{\lambda}^{(2)} = \\
            & (n - 1)^2 s_{(k)}^2 + 2 (n - 1) (n + k - 1) s_{(k)} s_{(k-1,1)} + (n + k - 1)^2 s_{(k-1,1)}^2 \\
            & - (n - 1) (n + k - 1) s_{\lambda} s_{(k-1)} - \frac{1}{2} (n + k - 1) (n + k - 2) s_{\lambda} s_{(k-2,1)} \, .
        \end{align*}
        By Lemma \ref{lemma:type 2&3}, we have that $s_{(k-1,1)}^2 - s_{\lambda} s_{(k-2,1)}$ is Schur-positive, which implies that
        \[
            (n + k - 1)^2 s_{(k-1,1)}^2 - \frac{1}{2} (n + k - 1) (n + k - 2) s_{\lambda} s_{(k-2,1)}
        \]
        is Schur positive.
        The sum of the other three terms in $(s_{\lambda}^{(1)})^2 - s_{\lambda}s_{\lambda}^{(2)}$ is almost Schur positive by the following, which is obtained by taking conjugate partitions in the previous case.
        \begin{align*}
            s_{(k)}^2 &
            = s_{(k,k)} + s_{(k+1,k-1)} + \cdots + s_{(2k-1,1)} + s_{(2k)} \\
            s_{(k)} s_{(k-1,1)} &
            = s_{(k,k-1,1)} + s_{(k+1,k-2,1)} + \cdots + s_{(2k-2,1,1)} \\ &
            \quad + s_{(k+1,k-1)} + \cdots + s_{(2k-1,1)} \\
            s_{\lambda} s_{(k-1)} &
            = s_{(k,k-1,1)} + s_{(k+1,k-2,1)} + \cdots + s_{(2k-2,1,1)} \\ &
            \quad + s_{(k,k)} + s_{(k+1,k-1)} + \cdots + s_{(2k-1,1)}
        \end{align*}
        The only problematic coefficient is that of $s_{(k,k)}$, but note that $c^{(k,k)}_{(k-1,1),(k-1,1)} = 1$ and $c^{(k,k)}_{\lambda,(k-2,1)} = 0$.
        Thus, the coefficient of $s_{(k,k)}$ in $(s_{\lambda}^{(1)})^2 - s_{\lambda}s_{\lambda}^{(2)}$ is given by
        \[
            (n - 1)^2 - (n - 1)(n + k - 1) + (n + k - 1)^2 \, ,
        \]
        which is positive.
        Therefore, $(s_{\lambda}^{(1)})^2 - s_{\lambda}s_{\lambda}^{(2)}$ is Schur positive.
    \end{itemize}
\end{proof}

%% file: 5_kk1.tex
\section{Partitions of the form $(k , k , 1)$}

Fix a positive integers $k \geq 3$
and consider the partition $\lambda = (k , k , 1)$.
Let $n \geq \ell(\lambda) = 3$ be the number of variables.
The goal of this section is to prove the following.

\begin{theorem} \label{theorem:kk1}
    Fix a positive integers $k \geq 3$,
    and let $\lambda= (k , k , 1)$.
    Then $(s_{\lambda}^{(1)})^2 - s_{\lambda}s_{\lambda}^{(2)}$ is Schur positive.
\end{theorem}

There are two partitions of $|\lambda| - 1$ contained in $\lambda$,
given by
$\lambda_{(2)} = (k , k - 1 , 1)$ and $\lambda_{(3)} = (k , k)$,
and three partitions of $|\lambda| - 2$ contained in $\lambda$,
given by
$\lambda_{(2,3)} = (k , k - 1)$ and $\lambda_{(2,\uparrow)} = (k - 1 , k - 1 , 1)$ and
$\lambda_{(2,\leftarrow)} = (k , k - 2 , 1)$.
The theorem follows from the following two lemmas.

\begin{lemma} \label{lemma:kk1 type 1}
    Let $\alpha$ be a partition of $2|\lambda| - 2 = 4k$,
    containing $\lambda$ and not equal to $(2k , k , k)$ or $(k^4)$.
    Then we have $2 c^{\alpha}_{\lambda_{(2)} , \lambda_{(3)}} \geq c^{\alpha}_{\lambda , \lambda_{(2,3)}}$.
\end{lemma}

\begin{lemma} \label{lemma:kk1 special}
    Let $\beta_1 = (2k , k , k)$ and $\beta_2 = (k^4)$.
    Then we have, for $i = 1 , 2$,
    \[
        \Big(
        c^{\beta_i}_{\lambda_{(2)} , \lambda_{(3)}} ,
        c^{\beta_i}_{\lambda_{(2)} , \lambda_{(2)}} ,
        c^{\beta_i}_{\lambda_{(3)} , \lambda_{(3)}} ,
        c^{\beta_i}_{\lambda , \lambda_{(2,3)}} ,
        c^{\beta_i}_{\lambda , \lambda_{(2,\uparrow)}} ,
        c^{\beta_i}_{\lambda , \lambda_{(2,\leftarrow)}}
        \Big)
        =
        (0 , 1 , 1 , 1 , 0 , 0) \, .
    \]
\end{lemma}

\begin{proof}[Proof of Theorem \ref{theorem:kk1}]
    By Lemma \ref{lemma:kk1 type 1} and equation (\ref{equation:main}),
    the only possibly negative coefficients in the Schur expansion of $(s_{\lambda}^{(1)})^2 - s_{\lambda}s_{\lambda}^{(2)}$ are those of $s_{(2k , k , k)}$ and $s_{(k^4)}$.
    However, by Lemma \ref{lemma:kk1 special} and equation (\ref{equation:main}),
    both these coefficients are equal to
    \begin{align*}
        &
            - (n + k - 2)(n - 2)
            + \frac{1}{2}(n + k - 2)(n + k - 1)
            + \frac{1}{2}(n - 2)(n - 1)
        \\
        &
            + \frac{1}{2}(n + k - 2)(n + k - 3)
            + \frac{1}{2}(n - 2)(n - 3)
        \\
            = \,
        &
            - (n + k - 2)(n - 2)
            + (n + k - 2)^2
            + (n - 2)^2
        \\
            = \,
        &
            (n + k - 2)k
            + (n - 2)^2 \, ,
    \end{align*}
    which is positive.
\end{proof}

Before proving Lemma \ref{lemma:kk1 type 1},
we observe that it is obviously true when 
$c^{\alpha}_{\lambda , \lambda_{(2,3)}} = 0$, which occurs in the following case:

\begin{lemma} \label{lemma:kk1 LRT}
    Let $\alpha = (\alpha_i)_{i = 1}^{\infty}$ be a partition of $2|\lambda| - 2$
    with $\alpha_1 > 2k$, $\alpha_3 > k$, $\alpha_5 > 1$, or $\alpha_6 > 0$.
    Then we have $c^{\alpha}_{\lambda , \lambda_{(2,3)}} = 0$.
\end{lemma}

\begin{proof}
    \begin{figure}
        \caption{Positions where $\alpha$ cannot have cells in order for $c^{\alpha}_{\lambda , \lambda_{(2,3)}} > 0$\vspace{5pt}}
        \label{figure:lem 5.4}
        \begin{tikzpicture}[scale=0.4]
            \draw
                (-0.7 , -0.5) node {\tiny $1$}
                (-0.7 , -1.5) node {\tiny $2$}
                (-0.7 , -2.5) node {\tiny $3$}
                (-0.7 , -3.5) node {\tiny $4$}
                (-0.7 , -4.5) node {\tiny $5$}
                (-0.7 , -5.5) node {\tiny $6$}
                (0.5 , 0.5) node {\tiny $1$}
                (1.5 , 0.5) node {\tiny $2$}
                (2.55 , 0.5) node {\tiny $\cdots$}
                (3.55 , 0.5) node {\tiny $k$}
                (4.55 , 0.5) node {\tiny $k \!\! + \!\! 1$}
                (6.05 , 0.5) node {\tiny $\cdots$}
                (7.55 , 0.5) node {\tiny $2k$}
            ;

            \fill[gray]
                (0 , 0) -- (4 , 0) -- (4 , -2) -- (1 , -2) -- (1 , -3) -- (0 , -3) -- (0 , 0)
            ;

            \draw[red , thick]
                (8 , 0) -- (9 , 0) -- (9 , -1) -- (8 , -1) -- (8 , 0)
                (8 , 0) -- (9 , -1)
                (9 , 0) -- (8 , -1)
                (4 , -2) -- (5 , -2) -- (5 , -3) -- (4 , -3) -- (4 , -2)
                (4 , -2) -- (5 , -3)
                (4 , -3) -- (5 , -2)
                (1 , -4) -- (2 , -4) -- (2 , -5) -- (1 , -5) -- (1 , -4)
                (1 , -4) -- (2 , -5)
                (2 , -4) -- (1 , -5)
                (0 , -5) -- (1 , -5) -- (1 , -6) -- (0 , -6) -- (0 , -5)
                (0 , -5) -- (1 , -6)
                (1 , -5) -- (0 , -6)
            ;
            
            \draw[dashed]
                (0 , -3) -- (1 , -3) -- (1 , -2) -- (4 , -2) -- (4 , -4) -- (1 , -4) -- (1 , -5) -- (0 , -5) -- (0 , -3)
                (4 , 0) -- (8 , 0) -- (8 , -2) -- (4 , -2) -- (4 , 0)
            ;
        \end{tikzpicture}
    \end{figure}
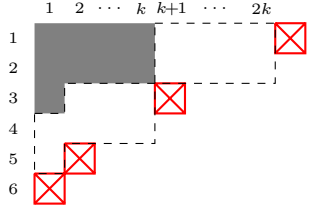
    
    If $\alpha_1 > 2k$,
    then there are more than $k$ cells in the first (non-empty) row of $\alpha / \lambda$
    (see Figure \ref{figure:lem 5.4}).
    However,
    we have $\lambda_{(2,3),1} = k$.
    Then, it follows from Lemma \ref{lemma:badLRT}(b) that $c^{\alpha}_{\lambda , \lambda_{(2,3)}} = 0$.
    
    The other three conditions would imply that
    there is a column in $\alpha / \lambda$ with at least $3$ cells
    (see Figure \ref{figure:lem 5.4}).
    However, we have $\ell(\lambda_{(2,3)}) = 2$.
    Thus, Lemma \ref{lemma:badLRT}(a) implies that $c^{\alpha}_{\lambda , \lambda_{(2,3)}} = 0$.
\end{proof}

\begin{proof}[Proof of Lemma \ref{lemma:kk1 type 1}]
    For each partition $\alpha$ of $2|\lambda| - 2$,
    containing $\lambda$ and not equal to $(2k , k , k)$ or $(k^4)$,
    in order to prove the inequality $2c^{\alpha}_{\lambda_{(2)} , \lambda_{(3)}} \geq c^{\alpha}_{\lambda , \lambda_{(2,3)}}$,
    we construct an at most 2-to-1 map from $\LR(\alpha/\lambda , \lambda_{(2,3)})$,
    the set of Littlewood-Richardson tableaux of shape $\alpha/\lambda$ and content $\lambda_{(2,3)}$,
    to $\LR(\alpha/\lambda_{(3)} , \lambda_{(2)})$,
    the set of Littlewood-Richardson tableaux of shape $\alpha/\lambda_{(3)}$ and content $\lambda_{(2)}$.
    
    By Lemma \ref{lemma:kk1 LRT},
    it suffices to consider $\alpha = (\alpha_i)_{i = 1}^{\infty}$ with $\alpha_1 \leq 2k$, $\alpha_3 \leq k$, $\alpha_5 \leq 1$ and $\alpha_6 = 0$.
    We divide into $3$ different cases:
    (1) $\alpha_4 = 0$ (2)  
    $\alpha_5 = 1$ and (3) $\alpha_4 > 0$ and $\alpha_5 = 0$.
    In each case,
    we give an algorithm that turns a Littlewood-Richardson tableau $T \in \LR(\alpha/\lambda , \lambda_{(2,3)})$
    into a Littlewood-Richardson tableau $T' \in \LR(\alpha/\lambda_{(3)} , \lambda_{(2)})$
    by adding an extra cell to the beginning of the third row and an extra integer $3$ to the content,
    and we check that $T'$ satisfies both conditions (a) and (b) in Definition \ref{definition:Littlewood-Richardson tableaux}.
    
    \begin{enumerate}[label=$\bullet$]
        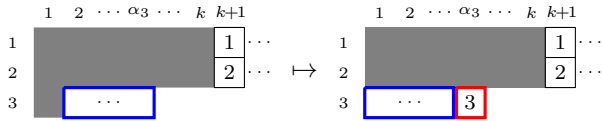
\begin{figure}
            \caption{The map $\LR(\alpha/\lambda , \lambda_{(2 , 3)}) \to \LR(\alpha/\lambda_{(3)} , \lambda_{(2)})$ in Case 1 \vspace{5pt}}
            \label{figure:kk1 case 1}
            \begin{tikzpicture}[scale=0.4]
                \draw
                    (-0.7 , -0.5) node {\tiny $1$}
                    (-0.7 , -1.5) node {\tiny $2$}
                    (-0.7 , -2.5) node {\tiny $3$}
                    (0.5 , 0.5) node {\tiny $1$}
                    (1.5 , 0.5) node {\tiny $2$}
                    (2.55 , 0.5) node {\tiny $\cdots$}
                    (3.5 , 0.5) node {\tiny $\alpha_3$}
                    (4.55 , 0.5) node {\tiny $\cdots$}
                    (5.55 , 0.5) node {\tiny $k$}
                    (6.55 , 0.5) node {\tiny $k \!\! + \!\! 1$}
                ;

                \fill[gray]
                    (0 , 0) -- (6 , 0) -- (6 , -2) -- (1 , -2) -- (1 , -3) -- (0 , -3) -- (0 , 0)
                ;

                \draw
                    (1 , -2) -- (4 , -2)
                    (1 , -3) -- (4 , -3)
                    (1 , -2) -- (1 , -3)
                    (4 , -2) -- (4 , -3)

                    (6 , 0) -- (7 , 0)
                    (6 , -1) -- (7 , -1)
                    (6 , -2) -- (7 , -2)
                    (6 , 0) -- (6 , -2)
                    (7 , 0) -- (7 , -2)

                    (7.55 , -0.5) node {\tiny $\cdots$}
                    (7.55 , -1.5) node {\tiny $\cdots$}
                    (2.55 , -2.5) node {\tiny $\cdots$}
                    (6.5 , -0.5) node {\footnotesize $1$}
                    (6.5 , -1.5) node {\footnotesize $2$}
                ;

                \draw[blue , very thick] (1 , -2) -- (4 , -2) -- (4 , -3) -- (1 , -3) -- (1 , -2) ;

                \draw (9 , -1.5) node {$\mapsto$} ;

                \draw
                    (10.3 , -0.5) node {\tiny $1$}
                    (10.3 , -1.5) node {\tiny $2$}
                    (10.3 , -2.5) node {\tiny $3$}
                    (11.5 , 0.5) node {\tiny $1$}
                    (12.5 , 0.5) node {\tiny $2$}
                    (13.55 , 0.5) node {\tiny $\cdots$}
                    (14.5 , 0.5) node {\tiny $\alpha_3$}
                    (15.55 , 0.5) node {\tiny $\cdots$}
                    (16.55 , 0.5) node {\tiny $k$}
                    (17.55 , 0.5) node {\tiny $k \!\! + \!\! 1$}
                ;

                \fill[gray]
                    (11 , 0) -- (17 , 0) -- (17 , -2) -- (11 , -2) -- (11 , 0)
                ;

                \draw
                    (11 , -2) -- (15 , -2)
                    (11 , -3) -- (15 , -3)
                    (11 , -2) -- (11 , -3)
                    (14 , -2) -- (14 , -3)
                    (15 , -2) -- (15 , -3)

                    (17 , 0) -- (18 , 0)
                    (17 , -1) -- (18 , -1)
                    (17 , -2) -- (18 , -2)
                    (17 , 0) -- (17 , -2)
                    (18 , 0) -- (18 , -2)

                    (18.55 , -0.5) node {\tiny $\cdots$}
                    (18.55 , -1.5) node {\tiny $\cdots$}
                    (12.55 , -2.5) node {\tiny $\cdots$}
                    (17.5 , -0.5) node {\footnotesize $1$}
                    (17.5 , -1.5) node {\footnotesize $2$}
                ;

                \draw[blue , very thick] (11 , -2) -- (13.95 , -2) -- (13.95 , -3) -- (11 , -3) -- (11 , -2) ;
                \draw[red , very thick] (14.05 , -2) -- (15 , -2) -- (15 , -3) -- (14.05 , -3) -- (14.05 , -2) ;
                \draw (14.5 , -2.5) node {\footnotesize $3$} ;
            \end{tikzpicture}
        \end{figure}
        
        \item \textbf{Case 1: } $\alpha_{4} = 0$, but $\alpha \neq (2k , k , k)$.
        \begin{enumerate}[label=$\circ$]
            \item $T'$ is obtained by 
            moving the third row of $T$ left by one cell,
            and placing the integer $3$ in the last cell of the third row.
            
            \item Then (a) is clear as $3$ is larger than all other integers in $T$.  For (b),note that there must be a $2$ in the second row as $\alpha \neq (2k , k , k)$.
            Thus, in the reverse word of $T'$,
            there is at least a $2$ in front of the added $3$.
        \end{enumerate}
        
        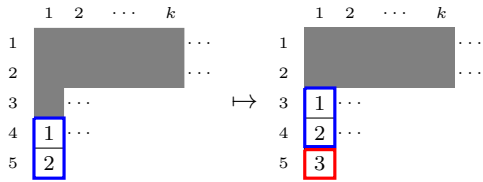
\begin{figure}
            \caption{The map $\LR(\alpha/\lambda , \lambda_{(2 , 3)}) \to \LR(\alpha/\lambda_{(3)} , \lambda_{(2)})$ in Case 2}
            \label{figure:kk1 case2}
            \begin{tikzpicture}[scale=0.4]
                \draw
                    (-0.7 , -0.5) node {\tiny $1$}
                    (-0.7 , -1.5) node {\tiny $2$}
                    (-0.7 , -2.5) node {\tiny $3$}
                    (-0.7 , -3.5) node {\tiny $4$}
                    (-0.7 , -4.5) node {\tiny $5$}
                    (0.5 , 0.5) node {\tiny $1$}
                    (1.5 , 0.5) node {\tiny $2$}
                    (3.05 , 0.5) node {\tiny $\cdots$}
                    (4.55 , 0.5) node {\tiny $k$}
                ;

                \fill[gray]
                    (0 , 0) -- (5 , 0) -- (5 , -2) -- (1 , -2) -- (1 , -3) -- (0 , -3) -- (0 , 0)
                ;

                \draw
                    (0 , -3) -- (1 , -3)
                    (0 , -4) -- (1 , -4)
                    (0 , -5) -- (1 , -5)
                    (0 , -3) -- (0 , -5)
                    (1 , -3) -- (1 , -5)

                    (1.55 , -2.5) node {\tiny $\cdots$}
                    (1.55 , -3.5) node {\tiny $\cdots$}
                    (5.55 , -0.5) node {\tiny $\cdots$}
                    (5.55 , -1.5) node {\tiny $\cdots$}
                    (0.5 , -3.5) node {\footnotesize $1$}
                    (0.5 , -4.5) node {\footnotesize $2$}
                ;

                \draw[blue , very thick] (0 , -3) -- (1 , -3) -- (1 , -5) -- (0 , -5) -- (0 , -3) ;

                \draw (7 , -2.5) node {$\mapsto$} ;

                \draw
                    (8.3 , -0.5) node {\tiny $1$}
                    (8.3 , -1.5) node {\tiny $2$}
                    (8.3 , -2.5) node {\tiny $3$}
                    (8.3 , -3.5) node {\tiny $4$}
                    (8.3 , -4.5) node {\tiny $5$}
                    (9.5 , 0.5) node {\tiny $1$}
                    (10.5 , 0.5) node {\tiny $2$}
                    (12.05 , 0.5) node {\tiny $\cdots$}
                    (13.55 , 0.5) node {\tiny $k$}
                ;

                \fill[gray]
                    (9 , 0) -- (14 , 0) -- (14 , -2) -- (9 , -2) -- (9 , 0)
                ;

                \draw
                    (9 , -2) -- (10 , -2)
                    (9 , -3) -- (10 , -3)
                    (9 , -4) -- (10 , -4)
                    (9 , -5) -- (10 , -5)
                    (9 , -2) -- (9 , -5)
                    (10 , -2) -- (10 , -5)

                    (10.55 , -2.5) node {\tiny $\cdots$}
                    (10.55 , -3.5) node {\tiny $\cdots$}
                    (14.55 , -0.5) node {\tiny $\cdots$}
                    (14.55 , -1.5) node {\tiny $\cdots$}
                    (9.5 , -2.5) node {\footnotesize $1$}
                    (9.5 , -3.5) node {\footnotesize $2$}
                ;

                \draw[blue , very thick] (9 , -2) -- (10 , -2) -- (10 , -3.95) -- (9 , -3.95) -- (9 , -2) ;
                \draw[red , very thick] (9 , -4.05) -- (10 , -4.05) -- (10 , -5) -- (9 , -5) -- (9 , -4.05) ;
                \draw (9.5 , -4.5) node {\footnotesize $3$} ;
            \end{tikzpicture}
        \end{figure}

        \item \textbf{Case 2: $\alpha_{5} = 1$.}
        \begin{enumerate}[label=$\circ$]
            \item 
            Note that the first column of $T$ must contain $1$ and $2$.
            Then $T'$ is obtained by moving each them one cell upward,
            and then placing the integer $3$ in the last cell of the first column.
            
            \item For (a) note first that the third row remains increasing as $1$ is the smallest integer in $T$. Moreover, the fourth row remains increasing, as all cells behind the first cell must contain $2$. For (b),            the integer $3$ is added to the end of the reverse reading word of $T$,
            so the result remains a lattice permutation.
        \end{enumerate}

        \item \textbf{Case 3: } $\alpha_{4} > 0$ and $\alpha_{5} = 0$, but $\alpha \neq (k^4)$.
        We further break this case down into three sub-cases.
        \begin{enumerate}[label=$-$]
            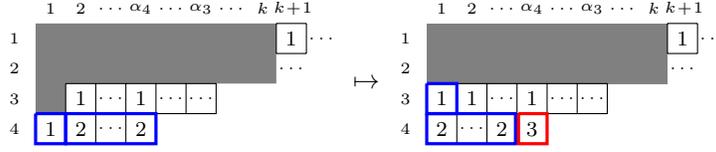
\begin{figure}
                \caption{The map $\LR(\alpha/\lambda , \lambda_{(2 , 3)}) \to \LR(\alpha/\lambda_{(3)} , \lambda_{(2)})$ in Case 3.1 \vspace{5pt}}
                \label{figure:kk1 case 3.1}
                \begin{tikzpicture}[scale=0.4]
                    \draw
                        (-0.7 , -0.5) node {\tiny $1$}
                        (-0.7 , -1.5) node {\tiny $2$}
                        (-0.7 , -2.5) node {\tiny $3$}
                        (-0.7 , -3.5) node {\tiny $4$}
                        (0.5 , 0.5) node {\tiny $1$}
                        (1.5 , 0.5) node {\tiny $2$}
                        (2.55 , 0.5) node {\tiny $\cdots$}
                        (3.5 , 0.5) node {\tiny $\alpha_4$}
                        (4.55 , 0.5) node {\tiny $\cdots$}
                        (5.5 , 0.5) node {\tiny $\alpha_3$}
                        (6.55 , 0.5) node {\tiny $\cdots$}
                        (7.55 , 0.5) node {\tiny $k$}
                        (8.55 , 0.5) node {\tiny $k \! + \! 1$}
                    ;

                    \fill[gray]
                        (0 , 0) -- (8 , 0) -- (8 , -2) -- (1 , -2) -- (1 , -3) -- (0 , -3) -- (0 , 0)
                    ;

                    \draw
                        (1 , -2) -- (6 , -2)
                        (0 , -3) -- (6 , -3)
                        (0 , -4) -- (4 , -4)
                        (0 , -3) -- (0 , -4)
                        (1 , -2) -- (1 , -4)
                        (2 , -2) -- (2 , -4)
                        (3 , -2) -- (3 , -4)
                        (4 , -2) -- (4 , -4)
                        (5 , -2) -- (5 , -3)
                        (6 , -2) -- (6 , -3)

                        (8 , 0) -- (9 , 0)
                        (8 , -1) -- (9 , -1)
                        (8 , 0) -- (8 , -1)
                        (9 , 0) -- (9 , -1)

                        (9.55 , -0.5) node {\tiny $\cdots$}
                        (8.55 , -1.5) node {\tiny $\cdots$}
                        (2.55 , -2.5) node {\tiny $\cdots$}
                        (4.55 , -2.5) node {\tiny $\cdots$}
                        (5.55 , -2.5) node {\tiny $\cdots$}
                        (2.55 , -3.5) node {\tiny $\cdots$}
                        (8.5 , -0.5) node {\footnotesize $1$}
                        (1.5 , -2.5) node {\footnotesize $1$}
                        (3.5 , -2.5) node {\footnotesize $1$}
                        (1.5 , -3.5) node {\footnotesize $2$}
                        (3.5 , -3.5) node {\footnotesize $2$}
                    ;

                    \draw (0.5 , -3.5) node {\footnotesize $1$} ;
                    \draw[blue , very thick] (0 , -3) -- (1 , -3) -- (1 , -4) -- (0 , -4) -- (0 , -3) ;
                    \draw[blue , very thick] (1 , -3) -- (4 , -3) -- (4 , -4) -- (1 , -4) -- (1 , -3) ;

                    \draw (11 , -2) node {$\mapsto$} ;

                    \draw
                        (12.3 , -0.5) node {\tiny $1$}
                        (12.3 , -1.5) node {\tiny $2$}
                        (12.3 , -2.5) node {\tiny $3$}
                        (12.3 , -3.5) node {\tiny $4$}
                        (13.5 , 0.5) node {\tiny $1$}
                        (14.5 , 0.5) node {\tiny $2$}
                        (15.55 , 0.5) node {\tiny $\cdots$}
                        (16.5 , 0.5) node {\tiny $\alpha_4$}
                        (17.55 , 0.5) node {\tiny $\cdots$}
                        (18.5 , 0.5) node {\tiny $\alpha_3$}
                        (19.55 , 0.5) node {\tiny $\cdots$}
                        (20.55 , 0.5) node {\tiny $k$}
                        (21.55 , 0.5) node {\tiny $k \! + \! 1$}
                    ;

                    \fill[gray]
                        (13 , 0) -- (21 , 0) -- (21 , -2) -- (13 , -2) -- (13 , 0)
                    ;

                    \draw
                        (13 , -2) -- (19 , -2)
                        (13 , -3) -- (19 , -3)
                        (13 , -4) -- (17 , -4)
                        (13 , -2) -- (13 , -4)
                        (14 , -2) -- (14 , -4)
                        (15 , -2) -- (15 , -4)
                        (16 , -2) -- (16 , -4)
                        (17 , -2) -- (17 , -4)
                        (18 , -2) -- (18 , -3)
                        (19 , -2) -- (19 , -3)

                        (21 , 0) -- (22 , 0)
                        (21 , -1) -- (22 , -1)
                        (21 , 0) -- (21 , -1)
                        (22 , 0) -- (22 , -1)

                        (22.55 , -0.5) node {\tiny $\cdots$}
                        (21.55 , -1.5) node {\tiny $\cdots$}
                        (15.55 , -2.5) node {\tiny $\cdots$}
                        (17.55 , -2.5) node {\tiny $\cdots$}
                        (18.55 , -2.5) node {\tiny $\cdots$}
                        (14.55 , -3.5) node {\tiny $\cdots$}
                        (21.5 , -0.5) node {\footnotesize $1$}
                        (14.5 , -2.5) node {\footnotesize $1$}
                        (16.5 , -2.5) node {\footnotesize $1$}
                        (13.5 , -3.5) node {\footnotesize $2$}
                        (15.5 , -3.5) node {\footnotesize $2$}
                    ;

                    \draw (13.5 , -2.5) node {\footnotesize $1$} ;
                    \draw[blue , very thick] (13 , -2) -- (14 , -2) -- (14 , -3) -- (13 , -3) -- (13 , -2) ;
                    \draw[blue , very thick] (13 , -3) -- (15.95 , -3) -- (15.95 , -4) -- (13 , -4) -- (13 , -3) ;
                    \draw[red , very thick] (16.05 , -3) -- (17 , -3) -- (17 , -4) -- (16.05 , -4) -- (16.05 , -3) ;
                    \draw (16.5 , -3.5) node {\footnotesize $3$} ;
                \end{tikzpicture}
            \end{figure}
            
            \item \textbf{Case 3.1:} The first cell of the fourth row contains $1$.
            \begin{enumerate}[label=$\circ$]
                \item $T'$ is obtained by
                moving the $1$ in the first cell of the fourth row one cell upward,
                and moving the rest of the fourth row left by one cell,
                and placing the integer $3$ in the last cell of the fourth row.
                
                \item Condition (a) is then clear.  For (b), note that there is at least one $2$ in the third row,
                since $\alpha$ would have to be $(k^4)$ for all $k - 1$ of $2$'s to be in the fourth row.
                Thus, in the reverse word of $T'$, there is at least one $2$ in front of the added $3$.
            \end{enumerate}
            
            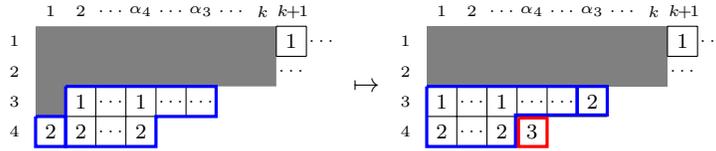
\begin{figure}
                \caption{The map $\LR(\alpha/\lambda , \lambda_{(2 , 3)}) \to \LR(\alpha/\lambda_{(3)} , \lambda_{(2)})$ in Case 3.2 \vspace{5pt}}
                \label{figure:kk1 case 3.2}
                \begin{tikzpicture}[scale=0.4]
                    \draw
                        (-0.7 , -0.5) node {\tiny $1$}
                        (-0.7 , -1.5) node {\tiny $2$}
                        (-0.7 , -2.5) node {\tiny $3$}
                        (-0.7 , -3.5) node {\tiny $4$}
                        (0.5 , 0.5) node {\tiny $1$}
                        (1.5 , 0.5) node {\tiny $2$}
                        (2.55 , 0.5) node {\tiny $\cdots$}
                        (3.5 , 0.5) node {\tiny $\alpha_4$}
                        (4.55 , 0.5) node {\tiny $\cdots$}
                        (5.5 , 0.5) node {\tiny $\alpha_3$}
                        (6.55 , 0.5) node {\tiny $\cdots$}
                        (7.55 , 0.5) node {\tiny $k$}
                        (8.55 , 0.5) node {\tiny $k \!\! + \!\! 1$}
                    ;

                    \fill[gray]
                        (0 , 0) -- (8 , 0) -- (8 , -2) -- (1 , -2) -- (1 , -3) -- (0 , -3) -- (0 , 0)
                    ;

                    \draw
                        (1 , -2) -- (6 , -2)
                        (0 , -3) -- (6 , -3)
                        (0 , -4) -- (4 , -4)
                        (0 , -3) -- (0 , -4)
                        (1 , -2) -- (1 , -4)
                        (2 , -2) -- (2 , -4)
                        (3 , -2) -- (3 , -4)
                        (4 , -2) -- (4 , -4)
                        (5 , -2) -- (5 , -3)
                        (6 , -2) -- (6 , -3)

                        (8 , 0) -- (9 , 0)
                        (8 , -1) -- (9 , -1)
                        (8 , 0) -- (8 , -1)
                        (9 , 0) -- (9 , -1)

                        (9.55 , -0.5) node {\tiny $\cdots$}
                        (8.55 , -1.5) node {\tiny $\cdots$}
                        (2.55 , -2.5) node {\tiny $\cdots$}
                        (4.55 , -2.5) node {\tiny $\cdots$}
                        (5.55 , -2.5) node {\tiny $\cdots$}
                        (2.55 , -3.5) node {\tiny $\cdots$}
                        (8.5 , -0.5) node {\footnotesize $1$}
                        (1.5 , -2.5) node {\footnotesize $1$}
                        (3.5 , -2.5) node {\footnotesize $1$}
                        (1.5 , -3.5) node {\footnotesize $2$}
                        (3.5 , -3.5) node {\footnotesize $2$}
                    ;

                    \draw (0.5 , -3.5) node {\footnotesize $2$} ;
                    \draw[blue , very thick] (0 , -3) -- (1 , -3) -- (1 , -4) -- (0 , -4) -- (0 , -3) ;
                    \draw[blue , very thick] (1 , -2) -- (6 , -2) -- (6 , -3) -- (4 , -3) -- (4 , -4) -- (1 , -4) -- (1 , -2) ;
                    \draw (11 , -2) node {$\mapsto$} ;

                    \draw
                        (12.3 , -0.5) node {\tiny $1$}
                        (12.3 , -1.5) node {\tiny $2$}
                        (12.3 , -2.5) node {\tiny $3$}
                        (12.3 , -3.5) node {\tiny $4$}
                        (13.5 , 0.5) node {\tiny $1$}
                        (14.5 , 0.5) node {\tiny $2$}
                        (15.55 , 0.5) node {\tiny $\cdots$}
                        (16.5 , 0.5) node {\tiny $\alpha_4$}
                        (17.55 , 0.5) node {\tiny $\cdots$}
                        (18.5 , 0.5) node {\tiny $\alpha_3$}
                        (19.55 , 0.5) node {\tiny $\cdots$}
                        (20.55 , 0.5) node {\tiny $k$}
                        (21.55 , 0.5) node {\tiny $k \!\! + \!\! 1$}
                    ;

                    \fill[gray]
                        (13 , 0) -- (21 , 0) -- (21 , -2) -- (13 , -2) -- (13 , 0)
                    ;

                    \draw
                        (13 , -2) -- (19 , -2)
                        (13 , -3) -- (19 , -3)
                        (13 , -4) -- (17 , -4)
                        (13 , -2) -- (13 , -4)
                        (14 , -2) -- (14 , -4)
                        (15 , -2) -- (15 , -4)
                        (16 , -2) -- (16 , -4)
                        (17 , -2) -- (17 , -4)
                        (18 , -2) -- (18 , -3)
                        (19 , -2) -- (19 , -3)

                        (21 , 0) -- (22 , 0)
                        (21 , -1) -- (22 , -1)
                        (21 , 0) -- (21 , -1)
                        (22 , 0) -- (22 , -1)

                        (22.55 , -0.5) node {\tiny $\cdots$}
                        (21.55 , -1.5) node {\tiny $\cdots$}
                        (14.55 , -2.5) node {\tiny $\cdots$}
                        (16.55 , -2.5) node {\tiny $\cdots$}
                        (17.55 , -2.5) node {\tiny $\cdots$}
                        (14.55 , -3.5) node {\tiny $\cdots$}
                        (21.5 , -0.5) node {\footnotesize $1$}
                        (13.5 , -2.5) node {\footnotesize $1$}
                        (15.5 , -2.5) node {\footnotesize $1$}
                        (13.5 , -3.5) node {\footnotesize $2$}
                        (15.5 , -3.5) node {\footnotesize $2$}
                    ;

                    \draw (18.5 , -2.5) node {\footnotesize $2$} ;
                    \draw[blue , very thick] (18 , -2) -- (19 , -2) -- (19 , -2.95) -- (18 , -2.95) -- (18 , -2) ;
                    \draw[blue , very thick] (13 , -2) -- (18 , -2) -- (18 , -2.95) -- (15.95 , -2.95) -- (15.95 , -4) -- (13 , -4) -- (13 , -2) ;
                    \draw[red , very thick] (16.05 , -3.05) -- (17 , -3.05) -- (17 , -4) -- (16.05 , -4) -- (16.05 , -3.05) ;
                    \draw (16.5 , -3.5) node {\footnotesize $3$} ;
                \end{tikzpicture}
            \end{figure}
            
            \item \textbf{Case 3.2: } The first cell in the fourth row contains $2$, and there are fewer $2$'s in the second and third rows than the $1$'s in the first row.
            \begin{enumerate}[label=$\circ$]
                \item Then $T'$ is obtained by
                moving ove the third row left by one cell,
                moving the $2$ in the first cell of the fourth row to the last cell of the third row,
                moving the rest of the fourth row left by one cell,
                and placing the integer $3$ in the last cell of the fourth row.
                
                \item Condition (a) is clear.  For (b), 
                note that the number of $2$'s in the second and third rows is no more than the number of $1$'s in the first row by assumption.
                Also, in the reverse reading work of $T'$, the $2$ moved to the third row is in front of the added $3$.
            \end{enumerate}
            
            \begin{figure}
                \caption{The map $\LR(\alpha/\lambda , \lambda_{(2 , 3)}) \to \LR(\alpha/\lambda_{(3)} , \lambda_{(2)})$ in Case 3.3 \vspace{5pt}}
                \label{figure:kk1 case 3.3}
                \begin{tikzpicture}[scale=0.4]
                    \draw
                        (-0.7 , -0.5) node {\tiny $1$}
                        (-0.7 , -1.5) node {\tiny $2$}
                        (-0.7 , -2.5) node {\tiny $3$}
                        (-0.7 , -3.5) node {\tiny $4$}
                        (0.5 , 0.5) node {\tiny $1$}
                        (1.55 , 0.5) node {\tiny $\cdots$}
                        (2.5 , 0.5) node {\tiny $\alpha_4$}
                        (3.5 , 0.5) node {\tiny $\alpha_4 \!\! + \!\! 1$}
                        (4.55 , 0.5) node {\tiny $\cdots$}
                        (5.5 , 0.5) node {\tiny $\alpha_3$}
                        (6.55 , 0.5) node {\tiny $\cdots$}
                        (7.55 , 0.5) node {\tiny $k$}
                        (8.55 , 0.5) node {\tiny $k \!\! + \!\! 1$}
                    ;

                    \fill[gray]
                        (0 , 0) -- (8 , 0) -- (8 , -2) -- (1 , -2) -- (1 , -3) -- (0 , -3) -- (0 , 0)
                    ;

                    \draw
                        (1 , -2) -- (6 , -2)
                        (0 , -3) -- (6 , -3)
                        (0 , -4) -- (3 , -4)
                        (0 , -3) -- (0 , -4)
                        (1 , -2) -- (1 , -4)
                        (2 , -2) -- (2 , -4)
                        (3 , -2) -- (3 , -4)
                        (4 , -2) -- (4 , -3)
                        (5 , -2) -- (5 , -3)
                        (6 , -2) -- (6 , -3)

                        (8 , 0) -- (9 , 0)
                        (8 , -1) -- (9 , -1)
                        (8 , -2) -- (9 , -2)
                        (8 , 0) -- (8 , -2)
                        (9 , 0) -- (9 , -2)

                        (9.55 , -0.5) node {\tiny $\cdots$}
                        (9.55 , -1.5) node {\tiny $\cdots$}
                        (2.55 , -2.5) node {\tiny $\cdots$}
                        (4.55 , -2.5) node {\tiny $\cdots$}
                        (5.55 , -2.5) node {\tiny $\cdots$}
                        (1.55 , -3.5) node {\tiny $\cdots$}
                        (8.5 , -0.5) node {\footnotesize $1$}
                        (1.5 , -2.5) node {\footnotesize $1$}
                        (3.5 , -2.5) node {\footnotesize $1$}
                        (0.5 , -3.5) node {\footnotesize $2$}
                        (2.5 , -3.5) node {\footnotesize $2$}
                        (8.5 , -1.5) node {\footnotesize $2$}
                    ;

                    \draw[blue , very thick] (1 , -2) -- (6 , -2) -- (6 , -3) -- (1 , -3) -- (1 , -2) ;

                    \draw (11 , -2) node {$\mapsto$} ;

                    \draw
                        (12.3 , -0.5) node {\tiny $1$}
                        (12.3 , -1.5) node {\tiny $2$}
                        (12.3 , -2.5) node {\tiny $3$}
                        (12.3 , -3.5) node {\tiny $4$}
                        (13.5 , 0.5) node {\tiny $1$}
                        (14.55 , 0.5) node {\tiny $\cdots$}
                        (15.5 , 0.5) node {\tiny $\alpha_4$}
                        (16.5 , 0.5) node {\tiny $\alpha_4 \!\! + \!\! 1$}
                        (17.55 , 0.5) node {\tiny $\cdots$}
                        (18.5 , 0.5) node {\tiny $\alpha_3$}
                        (19.55 , 0.5) node {\tiny $\cdots$}
                        (20.55 , 0.5) node {\tiny $k$}
                        (21.55 , 0.5) node {\tiny $k \!\! + \!\! 1$}
                    ;

                    \fill[gray]
                        (13 , 0) -- (21 , 0) -- (21 , -2) -- (13 , -2) -- (13 , 0)
                    ;

                    \draw
                        (13 , -2) -- (19 , -2)
                        (13 , -3) -- (19 , -3)
                        (13 , -4) -- (16 , -4)
                        (13 , -2) -- (13 , -4)
                        (14 , -2) -- (14 , -4)
                        (15 , -2) -- (15 , -4)
                        (16 , -2) -- (16 , -4)
                        (17 , -2) -- (17 , -3)
                        (18 , -2) -- (18 , -3)
                        (19 , -2) -- (19 , -3)

                        (21 , 0) -- (22 , 0)
                        (21 , -1) -- (22 , -1)
                        (21 , -2) -- (22 , -2)
                        (21 , 0) -- (21 , -2)
                        (22 , 0) -- (22 , -2)

                        (22.55 , -0.5) node {\tiny $\cdots$}
                        (22.55 , -1.5) node {\tiny $\cdots$}
                        (14.55 , -2.5) node {\tiny $\cdots$}
                        (16.55 , -2.5) node {\tiny $\cdots$}
                        (17.55 , -2.5) node {\tiny $\cdots$}
                        (14.55 , -3.5) node {\tiny $\cdots$}
                        (21.5 , -0.5) node {\footnotesize $1$}
                        (21.5 , -1.5) node {\footnotesize $2$}
                        (13.5 , -2.5) node {\footnotesize $1$}
                        (15.5 , -2.5) node {\footnotesize $1$}
                        (13.5 , -3.5) node {\footnotesize $2$}
                        (15.5 , -3.5) node {\footnotesize $2$}
                    ;

                    \draw[blue , very thick] (13 , -2) -- (17.95 , -2) -- (17.95 , -3) -- (13 , -3) -- (13 , -2) ;
                    \draw[red , very thick] (18.05 , -2) -- (19 , -2) -- (19 , -3) -- (18.05 , -3) -- (18.05 , -2) ;
                    \draw (18.5 , -2.5) node {\footnotesize $3$} ;
                \end{tikzpicture}
            \end{figure}
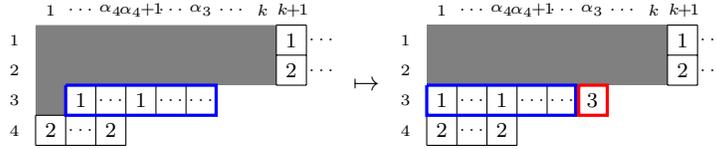

            \item \textbf{Case 3.3: }The first cell in the fourth row contains $2$, and there are as many $2$'s in the second and third rows as the $1$'s in the first row.
            \begin{enumerate}[label=$\circ$]
                \item Then $T'$ is obtained by 
                moving the third row left by one cell,
                and placing the integer $3$ in the last cell of the third row.
                
                \item For condition (a) note                the fourth row of $T$ contains only $2$'s,
                so there are $k - 1 - \alpha_4$ of $2$'s in the second and the third rows of $T$
                and the same amount of $1$'s in the first row of $T$.
                Thus, there are $k - (k - 1 - \alpha_4) = \alpha_4 + 1$ of $1$'s in the third row of $T$,
                which ensures that the $\alpha_4$-th column of $T'$ is strictly increasing.
                
                The check for condition (b) is similar to before.  As argued above, there are $\alpha_4 + 1$ of $1$'s in the third row of $T$, which implies that there are at most $(k - 1) - (\alpha_4 + 1) = k - 2 - \alpha_4$ of $2$'s in the third row of $T$.
                Therefore, there is at least one $2$ in the second row of $2$,
                and hence in the reverse reading work of $T'$,
                there is at least one $2$ in front of the added $3$.
            \end{enumerate}
        \end{enumerate}
    \end{enumerate}
    
    \begin{figure}
        \captionsetup{width=280pt}
        \caption{Two distinct $T_1 , T_2 \in \LR((4,3,3,2)/(3,3,1) , (3,2))$ mapped to the same $T' \in \LR((4,3,3,2)/(3,3) , (3,2,1))$ \vspace{5pt}}
        \label{figure:kk1 exmp}
        \[
            \begin{tikzpicture}[scale = 0.4]
                \fill[gray] (0 , 0) -- (3 , 0) -- (3 , -2) -- (1 , -2) -- (1 , -3) -- (0 , -3) -- (0 , 0) ;
                
                \draw
                    (3 , 0) -- (4 , 0)
                    (3 , -1) -- (4 , -1)
                    (1 , -2) -- (3 , -2)
                    (0 , -3) -- (3 , -3)
                    (0 , -4) -- (2 , -4)
                    (0 , -3) -- (0 , -4)
                    (1 , -2) -- (1 , -4)
                    (2 , -2) -- (2 , -4)
                    (3 , 0) -- (3 , -1)
                    (3 , -2) -- (3 , -3)
                    (4 , 0) -- (4 , -1)

                    (3.5 , -0.5) node {\footnotesize $1$}
                    (1.5 , -2.5) node {\footnotesize $1$}
                    (2.5 , -2.5) node {\footnotesize $2$}
                    (0.5 , -3.5) node {\footnotesize $1$}
                    (1.5 , -3.5) node {\footnotesize $2$}
                ;
            \end{tikzpicture}
            \quad\quad , \quad\quad
            \begin{tikzpicture}[scale = 0.4]
                \fill[gray] (0 , 0) -- (3 , 0) -- (3 , -2) -- (1 , -2) -- (1 , -3) -- (0 , -3) -- (0 , 0) ;
                
                \draw
                    (3 , 0) -- (4 , 0)
                    (3 , -1) -- (4 , -1)
                    (1 , -2) -- (3 , -2)
                    (0 , -3) -- (3 , -3)
                    (0 , -4) -- (2 , -4)
                    (0 , -3) -- (0 , -4)
                    (1 , -2) -- (1 , -4)
                    (2 , -2) -- (2 , -4)
                    (3 , 0) -- (3 , -1)
                    (3 , -2) -- (3 , -3)
                    (4 , 0) -- (4 , -1)

                    (3.5 , -0.5) node {\footnotesize $1$}
                    (1.5 , -2.5) node {\footnotesize $1$}
                    (2.5 , -2.5) node {\footnotesize $1$}
                    (0.5 , -3.5) node {\footnotesize $2$}
                    (1.5 , -3.5) node {\footnotesize $2$}
                ;
            \end{tikzpicture}
            \quad\quad \raisebox{25pt}{$\mapsto$} \quad\quad
            \begin{tikzpicture}[scale = 0.4]
                \fill[gray] (0 , 0) -- (3 , 0) -- (3 , -2) -- (0 , -2) -- (0 , 0) ;
                
                \draw
                    (3 , 0) -- (4 , 0)
                    (3 , -1) -- (4 , -1)
                    (0 , -2) -- (3 , -2)
                    (0 , -3) -- (3 , -3)
                    (0 , -4) -- (2 , -4)
                    (0 , -2) -- (0 , -4)
                    (1 , -2) -- (1 , -4)
                    (2 , -2) -- (2 , -4)
                    (3 , 0) -- (3 , -1)
                    (3 , -2) -- (3 , -3)
                    (4 , 0) -- (4 , -1)

                    (3.5 , -0.5) node {\footnotesize $1$}
                    (0.5 , -2.5) node {\footnotesize $1$}
                    (1.5 , -2.5) node {\footnotesize $1$}
                    (2.5 , -2.5) node {\footnotesize $2$}
                    (0.5 , -3.5) node {\footnotesize $2$}
                    (1.5 , -3.5) node {\footnotesize $3$}
                ;
            \end{tikzpicture}
        \]
    \end{figure}
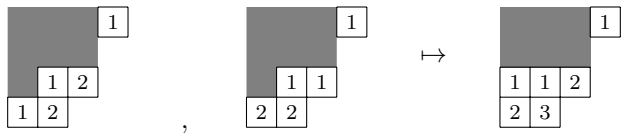
    
    Now the maps $\LR(\alpha/\lambda , \lambda_{(2,3)}) \to \LR(\alpha/\lambda_{(2)} , \lambda_{(3)})$ in the first two cases are clearly injective.    In Case 3, note that the $T'$ obtained in Case 3.3 is different from the other two subcases
    since the integer $3$ is placed in a different row.
    Thus, the maps in case 3 are at most 2-to-1.
    (In fact it is in general 2-to-1 as Figure \ref{figure:kk1 exmp} shows an example of two distinct $T$ mapped to the same $T'$.)
\end{proof}
	
\begin{proof}[Proof of Lemma \ref{lemma:kk1 special}]
    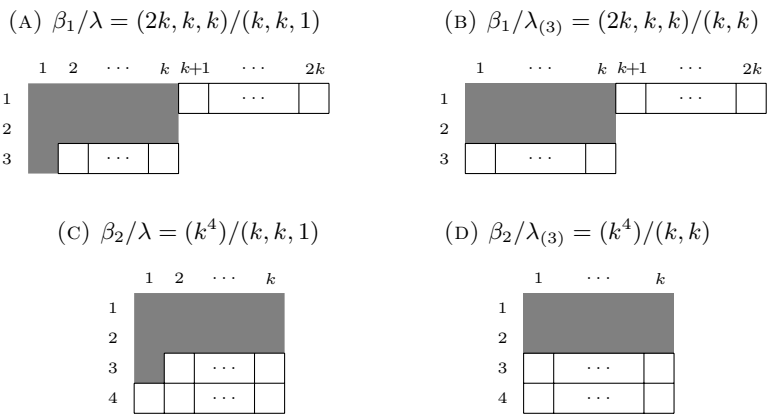
\begin{figure}
        \caption{Skew shapes $\beta_i / \lambda$ and $\beta_i / \lambda_{(3)}$}
        \label{figure:kk1 skew shapes}
        \begin{subfigure}{0.45\textwidth}
            \caption{$\beta_1 / \lambda = (2k , k , k) / (k , k , 1)$ \vspace{5pt}}
            \centering \begin{tikzpicture}[scale=0.4]
                \draw
                    (-0.7 , -0.5) node {\tiny $1$}
                    (-0.7 , -1.5) node {\tiny $2$}
                    (-0.7 , -2.5) node {\tiny $3$}
                    (0.5 , 0.5) node {\tiny $1$}
                    (1.5 , 0.5) node {\tiny $2$}
                    (3.05 , 0.5) node {\tiny $\cdots$}
                    (4.55 , 0.5) node {\tiny $k$}
                    (5.55 , 0.5) node {\tiny $k \!\! + \!\! 1$}
                    (7.55 , 0.5) node {\tiny $\cdots$}
                    (9.55 , 0.5) node {\tiny $2k$}
                ;

                \fill[gray]
                    (0 , 0) -- (5 , 0) -- (5 , -2) -- (1 , -2) -- (1 , -3) -- (0 , -3) -- (0 , 0)
                ;

                \draw
                    (5 , 0) -- (10 , 0)
                    (5 , -1) -- (10 , -1)
                    (5 , 0) -- (5 , -1)
                    (6 , 0) -- (6 , -1)
                    (9 , 0) -- (9 , -1)
                    (10 , 0) -- (10 , -1)

                    (1 , -2) -- (5 , -2)
                    (1 , -3) -- (5 , -3)
                    (1 , -2) -- (1 , -3)
                    (2 , -2) -- (2 , -3)
                    (4 , -2) -- (4 , -3)
                    (5 , -2) -- (5 , -3)

                    (3.05 , -2.5) node {\tiny $\cdots$}
                    (7.55 , -0.5) node {\tiny $\cdots$}
                ;
            \end{tikzpicture}
            \vspace{10pt}
        \end{subfigure}
        \begin{subfigure}{0.45\textwidth}
            \caption{$\beta_1 / \lambda_{(3)} = (2k , k , k) / (k , k)$ \vspace{5pt}}
            \centering \begin{tikzpicture}[scale=0.4]
                \draw
                    (-0.7 , -0.5) node {\tiny $1$}
                    (-0.7 , -1.5) node {\tiny $2$}
                    (-0.7 , -2.5) node {\tiny $3$}
                    (0.5 , 0.5) node {\tiny $1$}
                    (2.55 , 0.5) node {\tiny $\cdots$}
                    (4.55 , 0.5) node {\tiny $k$}
                    (5.55 , 0.5) node {\tiny $k \!\! + \!\! 1$}
                    (7.55 , 0.5) node {\tiny $\cdots$}
                    (9.55 , 0.5) node {\tiny $2k$}
                ;

                \fill[gray]
                    (0 , 0) -- (5 , 0) -- (5 , -2) -- (0 , -2) -- (0 , 0)
                ;

                \draw
                    (5 , 0) -- (10 , 0)
                    (5 , -1) -- (10 , -1)
                    (5 , 0) -- (5 , -1)
                    (6 , 0) -- (6 , -1)
                    (9 , 0) -- (9 , -1)
                    (10 , 0) -- (10 , -1)

                    (0 , -2) -- (5 , -2)
                    (0 , -3) -- (5 , -3)
                    (0 , -2) -- (0 , -3)
                    (1 , -2) -- (1 , -3)
                    (4 , -2) -- (4 , -3)
                    (5 , -2) -- (5 , -3)

                    (2.55 , -2.5) node {\tiny $\cdots$}
                    (7.55 , -0.5) node {\tiny $\cdots$}
                ;
            \end{tikzpicture}
            \vspace{10pt}
        \end{subfigure}
        \begin{subfigure}{0.4\textwidth}
            \caption{$\beta_2 / \lambda = (k^4) / (k , k , 1)$ \vspace{5pt}}
            \centering \begin{tikzpicture}[scale=0.4]
                \draw
                    (-0.7 , -0.5) node {\tiny $1$}
                    (-0.7 , -1.5) node {\tiny $2$}
                    (-0.7 , -2.5) node {\tiny $3$}
                    (-0.7 , -3.5) node {\tiny $4$}
                    (0.5 , 0.5) node {\tiny $1$}
                    (1.5 , 0.5) node {\tiny $2$}
                    (3.05 , 0.5) node {\tiny $\cdots$}
                    (4.55 , 0.5) node {\tiny $k$}
                ;

                \fill[gray]
                    (0 , 0) -- (5 , 0) -- (5 , -2) -- (1 , -2) -- (1 , -3) -- (0 , -3) -- (0 , 0)
                ;

                \draw
                    (1 , -2) -- (5 , -2)
                    (0 , -3) -- (5 , -3)
                    (0 , -4) -- (5 , -4)
                    (0 , -3) -- (0 , -4)
                    (1 , -2) -- (1 , -4)
                    (2 , -2) -- (2 , -4)
                    (4 , -2) -- (4 , -4)
                    (5 , -2) -- (5 , -4)

                    (3.05 , -2.5) node {\tiny $\cdots$}
                    (3.05 , -3.5) node {\tiny $\cdots$}
                ;
            \end{tikzpicture}
        \end{subfigure}
        \begin{subfigure}{0.4\textwidth}
            \caption{$\beta_2 / \lambda_{(3)} = (k^4) / (k , k)$ \vspace{5pt}}
            \centering \begin{tikzpicture}[scale=0.4]
                \draw
                    (-0.7 , -0.5) node {\tiny $1$}
                    (-0.7 , -1.5) node {\tiny $2$}
                    (-0.7 , -2.5) node {\tiny $3$}
                    (-0.7 , -3.5) node {\tiny $4$}
                    (0.5 , 0.5) node {\tiny $1$}
                    (2.55 , 0.5) node {\tiny $\cdots$}
                    (4.55 , 0.5) node {\tiny $k$}
                ;

                \fill[gray]
                    (0 , 0) -- (5 , 0) -- (5 , -2) -- (0 , -2) -- (0 , 0)
                ;

                \draw
                    (0 , -2) -- (5 , -2)
                    (0 , -3) -- (5 , -3)
                    (0 , -4) -- (5 , -4)
                    (0 , -2) -- (0 , -4)
                    (1 , -2) -- (1 , -4)
                    (4 , -2) -- (4 , -4)
                    (5 , -2) -- (5 , -4)

                    (2.55 , -2.5) node {\tiny $\cdots$}
                    (2.55 , -3.5) node {\tiny $\cdots$}
                ;
            \end{tikzpicture}
        \end{subfigure}
    \end{figure}
    
    We have $c^{\beta_i}_{\lambda_{(2)} , \lambda_{(3)}} = c^{\beta_i}_{\lambda , \lambda_{(2,\uparrow)}} = c^{\beta_i}_{\lambda , \lambda_{(2,\leftarrow)}} = 0$, for $i = 1 , 2$, by Lemma $\ref{lemma:badLRT}$(c):
    all four skew shapes $\beta_i / \lambda_{(3)} , \beta_i / \lambda$ ($i = 1 , 2$) have only two non-empty rows
    (see Figure \ref{figure:kk1 skew shapes}),
    but $\ell(\lambda_{(2)}) = \ell(\lambda_{(2,\uparrow)}) = \ell(\lambda_{(2,\leftarrow)}) = 3$,
    so $\LR(\beta_i / \lambda_{(3)} , \lambda_{(2)}) , \LR(\beta_i / \lambda , \lambda_{(2,\uparrow)}) , \LR(\beta_i / \lambda , \lambda_{(2,\leftarrow)})$, for $i = 1 , 2$, are all empty.
    
    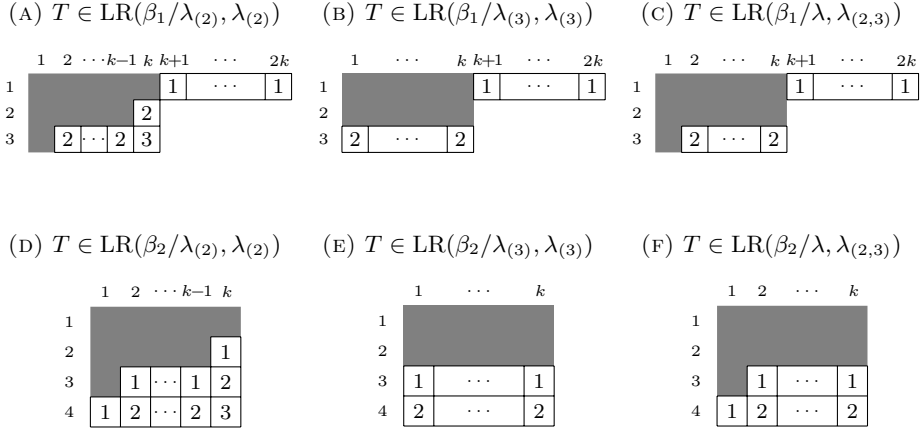
\begin{figure}
        \captionsetup{width=290pt}
        \caption{Littlewood-Richardson tableaux in $\LR(\beta_i / \lambda_{(2)} , \lambda_{(2)})$, $\LR(\beta_i / \lambda_{(3)} , \lambda_{(3)})$, $\LR(\beta_i / \lambda , \lambda_{(2 , 3)})$, $i = 1,2$}
        \label{figure:kk1 special}
        \begin{subfigure}{0.32\textwidth}
            \caption{$T \in \LR(\beta_1 / \lambda_{(2)} , \lambda_{(2)})$ \vspace{5pt}}
            \begin{tikzpicture}[scale=0.35]
                \draw
                    (-0.7 , -0.5) node {\tiny $1$}
                    (-0.7 , -1.5) node {\tiny $2$}
                    (-0.7 , -2.5) node {\tiny $3$}
                    (0.5 , 0.5) node {\tiny $1$}
                    (1.5 , 0.5) node {\tiny $2$}
                    (2.55 , 0.5) node {\tiny $\cdots$}
                    (3.55 , 0.5) node {\tiny $k \!\! - \!\! 1$}
                    (4.55 , 0.5) node {\tiny $k$}
                    (5.55 , 0.5) node {\tiny $k \!\! + \!\! 1$}
                    (7.55 , 0.5) node {\tiny $\cdots$}
                    (9.55 , 0.5) node {\tiny $2k$}
                ;

                \fill[gray]
                    (0 , 0) -- (5 , 0) -- (5 , -1) -- (4 , -1) -- (4 , -2) -- (1 , -2) -- (1 , -3) -- (0 , -3) -- (0 , 0)
                ;

                \draw
                    (5 , 0) -- (10 , 0)
                    (5 , -1) -- (10 , -1)
                    (5 , 0) -- (5 , -1)
                    (6 , 0) -- (6 , -1)
                    (9 , 0) -- (9 , -1)
                    (10 , 0) -- (10 , -1)

                    (4 , -1) -- (5 , -1)
                    (1 , -2) -- (5 , -2)
                    (1 , -3) -- (5 , -3)
                    (1 , -2) -- (1 , -3)
                    (2 , -2) -- (2 , -3)
                    (3 , -2) -- (3 , -3)
                    (4 , -1) -- (4 , -3)
                    (5 , -1) -- (5 , -3)

                    (2.55 , -2.5) node {\tiny $\cdots$}
                    (7.55 , -0.5) node {\tiny $\cdots$}
                    (1.5 , -2.5) node {\footnotesize $2$}
                    (3.5 , -2.5) node {\footnotesize $2$}
                    (4.5 , -2.5) node {\footnotesize $3$}
                    (4.5 , -1.5) node {\footnotesize $2$}
                    (5.5 , -0.5) node {\footnotesize $1$}
                    (9.5 , -0.5) node {\footnotesize $1$}
                ;
            \end{tikzpicture}
            \vspace{10pt}
        \end{subfigure}
        \begin{subfigure}{0.32\textwidth}
            \caption{$T \in \LR(\beta_1 / \lambda_{(3)} , \lambda_{(3)})$ \vspace{5pt}}
            \begin{tikzpicture}[scale=0.35]
                \draw
                    (-0.7 , -0.5) node {\tiny $1$}
                    (-0.7 , -1.5) node {\tiny $2$}
                    (-0.7 , -2.5) node {\tiny $3$}
                    (0.5 , 0.5) node {\tiny $1$}
                    (2.55 , 0.5) node {\tiny $\cdots$}
                    (4.55 , 0.5) node {\tiny $k$}
                    (5.55 , 0.5) node {\tiny $k \!\! + \!\! 1$}
                    (7.55 , 0.5) node {\tiny $\cdots$}
                    (9.55 , 0.5) node {\tiny $2k$}
                ;

                \fill[gray]
                    (0 , 0) -- (5 , 0) -- (5 , -2) -- (0 , -2) -- (0 , 0)
                ;

                \draw
                    (5 , 0) -- (10 , 0)
                    (5 , -1) -- (10 , -1)
                    (5 , 0) -- (5 , -1)
                    (6 , 0) -- (6 , -1)
                    (9 , 0) -- (9 , -1)
                    (10 , 0) -- (10 , -1)

                    (0 , -2) -- (5 , -2)
                    (0 , -3) -- (5 , -3)
                    (0 , -2) -- (0 , -3)
                    (1 , -2) -- (1 , -3)
                    (4 , -2) -- (4 , -3)
                    (5 , -2) -- (5 , -3)

                    (2.55 , -2.5) node {\tiny $\cdots$}
                    (7.55 , -0.5) node {\tiny $\cdots$}
                    (0.5 , -2.5) node {\footnotesize $2$}
                    (4.5 , -2.5) node {\footnotesize $2$}
                    (5.5 , -0.5) node {\footnotesize $1$}
                    (9.5 , -0.5) node {\footnotesize $1$}
                ;
            \end{tikzpicture}
            \vspace{10pt}
        \end{subfigure}
        \begin{subfigure}{0.32\textwidth}
            \caption{$T \in \LR(\beta_1 / \lambda , \lambda_{(2 , 3)})$ \vspace{5pt}}
            \begin{tikzpicture}[scale=0.35]
                \draw
                    (-0.7 , -0.5) node {\tiny $1$}
                    (-0.7 , -1.5) node {\tiny $2$}
                    (-0.7 , -2.5) node {\tiny $3$}
                    (0.5 , 0.5) node {\tiny $1$}
                    (1.5 , 0.5) node {\tiny $2$}
                    (3.05 , 0.5) node {\tiny $\cdots$}
                    (4.55 , 0.5) node {\tiny $k$}
                    (5.55 , 0.5) node {\tiny $k \!\! + \!\! 1$}
                    (7.55 , 0.5) node {\tiny $\cdots$}
                    (9.55 , 0.5) node {\tiny $2k$}
                ;

                \fill[gray]
                    (0 , 0) -- (5 , 0) -- (5 , -2) -- (1 , -2) -- (1 , -3) -- (0 , -3) -- (0 , 0)
                ;

                \draw
                    (5 , 0) -- (10 , 0)
                    (5 , -1) -- (10 , -1)
                    (5 , 0) -- (5 , -1)
                    (6 , 0) -- (6 , -1)
                    (9 , 0) -- (9 , -1)
                    (10 , 0) -- (10 , -1)

                    (1 , -2) -- (5 , -2)
                    (1 , -3) -- (5 , -3)
                    (1 , -2) -- (1 , -3)
                    (2 , -2) -- (2 , -3)
                    (4 , -2) -- (4 , -3)
                    (5 , -2) -- (5 , -3)

                    (3.05 , -2.5) node {\tiny $\cdots$}
                    (7.55 , -0.5) node {\tiny $\cdots$}
                    (1.5 , -2.5) node {\footnotesize $2$}
                    (4.5 , -2.5) node {\footnotesize $2$}
                    (5.5 , -0.5) node {\footnotesize $1$}
                    (9.5 , -0.5) node {\footnotesize $1$}
                ;
            \end{tikzpicture}
            \vspace{10pt}
        \end{subfigure}
        \begin{subfigure}{0.32\textwidth}
            \caption{$T \in \LR(\beta_2 / \lambda_{(2)} , \lambda_{(2)})$ \vspace{5pt}}
            \centering \begin{tikzpicture}[scale=0.4]
                \draw
                    (-0.7 , -0.5) node {\tiny $1$}
                    (-0.7 , -1.5) node {\tiny $2$}
                    (-0.7 , -2.5) node {\tiny $3$}
                    (-0.7 , -3.5) node {\tiny $4$}
                    (0.5 , 0.5) node {\tiny $1$}
                    (1.5 , 0.5) node {\tiny $2$}
                    (2.55 , 0.5) node {\tiny $\cdots$}
                    (3.55 , 0.5) node {\tiny $k \!\! - \!\! 1$}
                    (4.55 , 0.5) node {\tiny $k$}
                ;
                
                \fill[gray]
                    (0 , 0) -- (5 , 0) -- (5 , -1) -- (4 , -1) -- (4 , -2) -- (1 , -2) -- (1 , -3) -- (0 , -3) -- (0 , 0)
                ;

                \draw
                    (4 , -1) -- (5 , -1)
                    (1 , -2) -- (5 , -2)
                    (0 , -3) -- (5 , -3)
                    (0 , -4) -- (5 , -4)
                    (0 , -3) -- (0 , -4)
                    (1 , -2) -- (1 , -4)
                    (2 , -2) -- (2 , -4)
                    (3 , -2) -- (3 , -4)
                    (4 , -1) -- (4 , -4)
                    (5 , -1) -- (5 , -4)

                    (2.55 , -2.5) node {\tiny $\cdots$}
                    (2.55 , -3.5) node {\tiny $\cdots$}
                    (1.5 , -2.5) node {\footnotesize $1$}
                    (3.5 , -2.5) node {\footnotesize $1$}
                    (0.5 , -3.5) node {\footnotesize $1$}
                    (1.5 , -3.5) node {\footnotesize $2$}
                    (3.5 , -3.5) node {\footnotesize $2$}
                    (4.5 , -1.5) node {\footnotesize $1$}
                    (4.5 , -2.5) node {\footnotesize $2$}
                    (4.5 , -3.5) node {\footnotesize $3$}
                ;
            \end{tikzpicture}
        \end{subfigure}
        \begin{subfigure}{0.32\textwidth}
            \caption{$T \in \LR(\beta_2 / \lambda_{(3)} , \lambda_{(3)})$ \vspace{5pt}}
            \centering \begin{tikzpicture}[scale=0.4]
                \draw
                    (-0.7 , -0.5) node {\tiny $1$}
                    (-0.7 , -1.5) node {\tiny $2$}
                    (-0.7 , -2.5) node {\tiny $3$}
                    (-0.7 , -3.5) node {\tiny $4$}
                    (0.5 , 0.5) node {\tiny $1$}
                    (2.55 , 0.5) node {\tiny $\cdots$}
                    (4.55 , 0.5) node {\tiny $k$}
                ;

                \fill[gray]
                    (0 , 0) -- (5 , 0) -- (5 , -2) -- (0 , -2) -- (0 , 0)
                ;

                \draw
                    (0 , -2) -- (5 , -2)
                    (0 , -3) -- (5 , -3)
                    (0 , -4) -- (5 , -4)
                    (0 , -2) -- (0 , -4)
                    (1 , -2) -- (1 , -4)
                    (4 , -2) -- (4 , -4)
                    (5 , -2) -- (5 , -4)

                    (2.55 , -2.5) node {\tiny $\cdots$}
                    (2.55 , -3.5) node {\tiny $\cdots$}
                    (0.5 , -2.5) node {\footnotesize $1$}
                    (4.5 , -2.5) node {\footnotesize $1$}
                    (0.5 , -3.5) node {\footnotesize $2$}
                    (4.5 , -3.5) node {\footnotesize $2$}
                ;
            \end{tikzpicture}
        \end{subfigure}
        \begin{subfigure}{0.32\textwidth}
            \caption{$T \in \LR(\beta_2 / \lambda , \lambda_{(2 , 3)})$ \vspace{5pt}}
            \centering \begin{tikzpicture}[scale=0.4]
                \draw
                    (-0.7 , -0.5) node {\tiny $1$}
                    (-0.7 , -1.5) node {\tiny $2$}
                    (-0.7 , -2.5) node {\tiny $3$}
                    (-0.7 , -3.5) node {\tiny $4$}
                    (0.5 , 0.5) node {\tiny $1$}
                    (1.5 , 0.5) node {\tiny $2$}
                    (3.05 , 0.5) node {\tiny $\cdots$}
                    (4.55 , 0.5) node {\tiny $k$}
                ;

                \fill[gray]
                    (0 , 0) -- (5 , 0) -- (5 , -2) -- (1 , -2) -- (1 , -3) -- (0 , -3) -- (0 , 0)
                ;

                \draw
                    (1 , -2) -- (5 , -2)
                    (0 , -3) -- (5 , -3)
                    (0 , -4) -- (5 , -4)
                    (0 , -3) -- (0 , -4)
                    (1 , -2) -- (1 , -4)
                    (2 , -2) -- (2 , -4)
                    (4 , -2) -- (4 , -4)
                    (5 , -2) -- (5 , -4)

                    (3.05 , -2.5) node {\tiny $\cdots$}
                    (3.05 , -3.5) node {\tiny $\cdots$}
                    (1.5 , -2.5) node {\footnotesize $1$}
                    (4.5 , -2.5) node {\footnotesize $1$}
                    (0.5 , -3.5) node {\footnotesize $1$}
                    (1.5 , -3.5) node {\footnotesize $2$}
                    (4.5 , -3.5) node {\footnotesize $2$}
                ;
            \end{tikzpicture}
        \end{subfigure}
    \end{figure}
    
    One may check that each of
    $\LR(\beta_i / \lambda_{(2)} , \lambda_{(2)}) ,
    \LR(\beta_i / \lambda_{(3)} , \lambda_{(3)}) ,
    \LR(\beta_i / \lambda , \lambda_{(2,3)})$
    ($i = 1 , 2$)
    is a singleton,
    the only element of which is depicted in Figure \ref{figure:kk1 special}.
    Therefore, we have
    $c^{\beta_i}_{\lambda_{(2)} , \lambda_{(2)}} =
    c^{\beta_i}_{\lambda_{(3)} , \lambda_{(3)}} =
    c^{\beta_i}_{\lambda , \lambda_{(2,3)}} = 1$
    for $i = 1 , 2$.
\end{proof}

%% file: 6_conjugate.tex
\section{Conjugate partitions and partitions of the form $(3 , 2^{k - 1})$}

The conjugate partition of a partition $\lambda$,
denoted by $\lambda'$,
is defined to be the partition corresponding to the transpose of the Young diagram of $\lambda$.
The Littlewood-Richardson coefficients are preserved under conjugation
in the sense that $c^{\lambda'}_{\mu' \nu'} = c^{\lambda}_{\mu \nu}$
for all partitions $\lambda , \mu , \nu$.
(See, for example, \cite[Corollary 2, p.\ 62]{Ful97} for a combinatorial proof.)
The corners are also preserved under conjugation,
and we have
$J(\lambda') = \{ \lambda_j : j \in J(\lambda) \}$.
It is also easy to see that
$(\lambda_{(j)})' = \lambda'_{(\lambda_j)}$,
$(\lambda_{(j , k)})' = \lambda'_{(\lambda_j , \lambda_k)}$,
and that either $(\lambda_{(j , \uparrow)})' = \lambda'_{(\lambda_j , \leftarrow)}$
or neither partition exists.

Fix $k \geq 3$.
Now we consider the partition $\lambda' = (3 , 2^{k - 1})$,
which is the conjugates of the partition $\lambda = (k , k , 1)$
considered in the previous section.
The following lemmas follow immediately from Lemmas \ref{lemma:kk1 type 1} and \ref{lemma:kk1 special}.

\begin{lemma} \label{lemma:conj type 1}
    Let $\alpha'$ be a partition of $2|\lambda'| - 2 = 4k$,
    containing $\lambda'$ and not equal to $(2k , k , k)' = (3^k , 1^k)$ or $(k^4)' = (4^k)$.
    Then we have $2 c^{\alpha'}_{\lambda'_{(1)} , \lambda'_{(k)}} \geq c^{\alpha'}_{\lambda' , \lambda'_{(1,k)}}$.
\end{lemma}

\begin{lemma} \label{lemma:conj special}
    Let $\beta'_1 = (2k , k , k)' = (3^k , 1^k)$ and $\beta'_2 = (k^4)' = (4^k)$.
    Then we have, for $i = 1 , 2$,
    \[
        \Big(
        c^{\beta'_i}_{\lambda'_{(1)} , \lambda'_{(k)}} ,
        c^{\beta'_i}_{\lambda'_{(k)} , \lambda'_{(k)}} ,
        c^{\beta'_i}_{\lambda'_{(1)} , \lambda'_{(1)}} ,
        c^{\beta'_i}_{\lambda' , \lambda'_{(1,k)}} ,
        c^{\beta'_i}_{\lambda' , \lambda'_{(k,\uparrow)}} ,
        c^{\beta'_i}_{\lambda' , \lambda'_{(k,\leftarrow)}}
        \Big)
        =
        (0 , 1 , 1 , 1 , 0 , 0) \, .
    \]
\end{lemma}

As a result,
we obtain the following.

\begin{theorem} \label{theorem:conjugate}
    Fix a positive integers $k \geq 3$,
    and let $\lambda' = (3 , 2^{k - 1})$.
    Then $(s_{\lambda'}^{(1)})^2 - s_{\lambda'}s_{\lambda'}^{(2)}$ is Schur positive.
\end{theorem}

\begin{proof}
    By Lemma \ref{lemma:conj type 1} and equation (\ref{equation:main}),
    the only possibly negative coefficients of $(s_{\lambda'}^{(1)})^2 - s_{\lambda'}s_{\lambda'}^{(2)}$ are those of $s_{(3^k , 1^k)}$ and $s_{(4^k)}$.
    However, by Lemma \ref{lemma:conj special} and equation (\ref{equation:main}),
    both these coefficients are equal to
    \begin{align*}
        &
            - (n + 2)(n + 2 - k)
            + \frac{1}{2}(n + 2)(n + 3)
            + \frac{1}{2}(n + 2 - k)(n + 3 - k)
        \\
        &
            + \frac{1}{2}(n + 2)(n + 1)
            + \frac{1}{2}(n + 2 - k)(n + 1 - k)
        \\
            = \,
        &
            - (n + 2)(n + 2 - k)
            + (n + 2)^2
            + (n + 2 - k)^2
        \\
            = \,
        &
            (n + 2)k
            + (n + 2 - k)^2 \, ,
    \end{align*}
    which is positive.
\end{proof}